\documentclass[12pt]{amsart}

\usepackage{amsmath,amssymb,amsthm}
\usepackage{mathtools}
\usepackage{tikz-cd}
\usepackage{hyperref}
\usepackage{microtype}
\usepackage{enumitem}
\usepackage{booktabs}
\usepackage[margin=1.2in]{geometry}
\usepackage{cleveref}

\hypersetup{hidelinks}

\theoremstyle{plain}
\newtheorem{theorem}{Theorem}[section]
\newtheorem{proposition}[theorem]{Proposition}
\newtheorem{lemma}[theorem]{Lemma}
\newtheorem{corollary}[theorem]{Corollary}

\theoremstyle{definition}
\newtheorem{definition}[theorem]{Definition}
\newtheorem{example}[theorem]{Example}
\newtheorem{construction}[theorem]{Construction}

\theoremstyle{remark}
\newtheorem{remark}[theorem]{Remark}

\newtheorem{hypothesis}[theorem]{Hypothesis}

\DeclareMathOperator{\Aut}{Aut}
\DeclareMathOperator{\id}{id}
\DeclareMathOperator{\im}{im}

\newcommand{\Z}{\mathbb{Z}}
\newcommand{\RP}{\mathbb{R}P}
\newcommand{\CP}{\mathbb{C}P}

\newcommand{\Autp}{\Aut^{+}_{\bullet}}
\newcommand{\Autb}{\Aut_{\bullet}}
\newcommand{\Autf}{\Aut_{\mathrm{free}}}
\newcommand{\Ompin}{\Omega^{\mathrm{Pin}^{+}}}
\newcommand{\Ompinm}{\Omega^{\mathrm{Pin}^{-}}}
\newcommand{\Omspin}{\Omega^{\mathrm{spin}}}
\newcommand{\wt}[1]{\widetilde{#1}}
\newcommand{\Zw}{\Z^{w}}
\newcommand{\BO}{BO}

\begin{document}

\title[Homotopy self-equivalences]{%
  Homotopy Self-Equivalences of $\RP^2\times S^2$\\
  and a Braid for non-orientable $4$-Manifolds}

\author{Mehmetcik Pamuk}
\date{\today}

\begin{abstract}
We compute the group of based homotopy classes of self-homotopy equivalences of the
non-orientable $4$-manifold $\RP^2\times S^2$. It is an elementary abelian $2$-group of rank $5$.
The subgroup of classes whose based lift to the universal cover $S^2\times S^2$ has degree $+1$ has
rank $4$ and is isomorphic to the group of free homotopy classes. The group maps onto the group of
self-equivalences of the algebraic $2$-type, which has order $8$, with kernel
$\pi_4(\RP^2\times S^2)\cong(\Z/2)^2$ generated by pinch-fold maps; that these maps are independent is shown by
an adjunction argument. We then construct a non-orientable analogue of the Hambleton--Kreck braid of exact sequences in
dimension $4$, for non-orientable manifolds with spin universal cover, stated for finite fundamental group. We apply
the braid to $\RP^2\times S^2$ and to $\RP^4$, and for the Hambleton--Hillman manifold with fundamental group $(\Z/2)^2$ we determine the normal $1$-type and show that self-equivalences act on $\pi_2$ through diagonal sign matrices. For
$\RP^2\times S^2$ the pinch-fold maps are detected by twisted spin bordism, and, assuming facts from surgery theory
stated as hypotheses, exactly one of the three non-trivial pinch-fold maps is homotopic to a homeomorphism.
\end{abstract}

\maketitle

\section{Introduction}

The classification of closed topological $4$-manifolds with finite fundamental
group builds on Freedman's work \cite{Freed82,FQ90}. For non-orientable
manifolds with $\pi_1=\Z/2$ the classification up to homeomorphism is due to
Hambleton, Kreck and Teichner \cite{HKT94}; for related results with other
cyclic fundamental groups see Torres \cite{Tor26}. Baykur and Morgan
\cite{BM25} study non-orientable $4$-manifolds from the point of view of
broken Lefschetz fibrations. Manifolds whose universal cover is $S^2\times S^2$
are treated by Hambleton and Hillman \cite{HH23}.

Knowing the group of homotopy self-equivalences of a manifold $M$ is a
prerequisite for describing the homotopy-equivalence classes in a given
homeomorphism class, and for the surgery-theoretic study of the structure set.
For oriented closed $4$-manifolds with odd order fundamental group,
Hambleton and Kreck \cite{HK04} describe $\Autb(M)$ by a braid of exact
sequences built from spin and normal-$2$-type bordism. Their argument is
written for oriented manifolds, and we are not aware of a non-orientable
version in the literature.

For a based self-equivalence $f$ of $M$ let $s(f)=\pm1$ be the degree of the unique lift $\widetilde f$ of $f$ to the universal cover fixing a chosen lift of the basepoint, and write $\Aut^{+}_\bullet(M)=\ker s$.

Let us say what the braid is for. For a self-equivalence $f$ of $M$ fixing the basepoint, the class $\alpha(f)=[M,f,\sigma]-[M,\mathrm{id},\sigma]$ in twisted spin bordism over $M$ records whether $f$ can be undone by a $B$-bordism over $M$; the kernel of $\alpha$ consists of the classes arising from the surgery-theoretic group $\widetilde{\mathcal H}(M)$. The kernel of the map from $\Aut_\bullet^+(M)$ to the self-equivalences of the algebraic $2$-type $P$ is the image of a relative bordism group $\widehat\Omega^B_5(P,M)$, and $\alpha$ restricted to this kernel is the boundary map $\partial$. So the braid turns the computation of the part of $\Aut_\bullet(M)$ invisible to $P$ into a computation in bordism, which can be done with the Atiyah--Hirzebruch spectral sequence.

In this paper we first treat the simplest non-orientable case, $M=\RP^2\times S^2$, by direct
obstruction theory (\cref{thm:main}). We then formulate a non-orientable analogue of the braid in
dimension~$4$ (\cref{thm:braidNO}) and test it on the same example.
\begin{itemize}
  \item Orientations are replaced by twisted degrees and $\Z$ by $\Zw$ in Poincar\'e duality. For finite $\pi_1$ a
  self-map of twisted degree one that is the identity on $\pi_1$ is a homotopy equivalence (\cref{lem:PDup}).
  \item The $B$-structures on the normal bundle with fixed $B\pi$-component and fixed orientation form a torsor over $H^1(M;\Z/2)$, which is non-zero because $w_1(M)\ne0$. For a fixed structure the braid needs no further gauge data: pushing a twisted spin manifold forward along
  a self-equivalence is canonical (\cref{lem:canon,lem:transport}), and the map $\alpha$ of the braid is a crossed homomorphism (\cref{prop:crossed}).
  \item For $\RP^2\times S^2$ the twisted spin bordism group $\widehat\Omega^B_4$ has order $16$, and the two generators of the kernel
  $\pi_4(M)$ are detected by $\alpha$, one of them in the filtration-zero term $E^\infty_{0,4}$ (\cref{prop:detect}).
  \item For $\RP^4$ the braid gives $\Aut^+_\bullet(\RP^4)=1$ (\cref{prop:RP4}).
  \item The following statement is conditional. Under two hypotheses from surgery theory, which we do not prove (\cref{sec:hyp}), of the three non-trivial elements of $\ker(\Autb(M)\to\Autb(P_2(M)))$ exactly one is
  homotopic to a homeomorphism. Independently of these hypotheses, no non-trivial element of this kernel is represented by a diffeomorphism that preserves the fixed $B$-structure on the normal bundle (\cref{prop:nodiffeo}).
\end{itemize}
\Cref{thm:main} is proved by obstruction theory in \cref{sec:automorphisms}, and the four sequences of \cref{thm:braidNO} are proved in \cref{sec:proofs}.

\begin{theorem}\label{thm:main}
  Let $M=\RP^2\times S^2$ with a basepoint $x_0$. Then
  \[
    \Autb(M)\;\cong\;(\Z/2)^5 .
  \]
  Generators are $a=r_0\times\id$ with $r_0$ a reflection of $\RP^2$,
  $b=\id\times\rho$ with $\rho$ a reflection of $S^2$,
  $s_\psi([x],y)=([x],\psi([x])\,y)$ with $\psi\colon\RP^2\to SO(3)$ the
  non-trivial map, and two pinch-fold maps $k_1,k_2$, defined in
  \cref{sec:pinchfold}, corresponding to the generators of
  $\pi_4(M)\cong(\Z/2)^2$.
  The sign homomorphism $s\colon\Autb(M)\to\{\pm1\}$, given by the degree of
  the based lift to $\wt M=S^2\times S^2$, has kernel
  $\Autb^{+}(M)\cong(\Z/2)^4$ generated by $ab$, $s_\psi$, $k_1$, $k_2$.
\end{theorem}

\begin{corollary}\label{cor:free}
  The group of free homotopy classes of homotopy self-equivalences of $M$ is
  isomorphic to $\Autb^{+}(M)\cong(\Z/2)^4$.
\end{corollary}

Consider the manifold of Hambleton and Hillman. Their
manifold $N$ has $\pi_1(N)\cong(\Z/2)^2$, $\chi(N)=1$, and universal cover
homotopy equivalent to $S^2\times S^2$. It has two non-orientable double covers. One is $\RP^4\#_{S^1}\RP^4$, with $v_2\ne0$; the other has $v_2=0$ and is homotopy equivalent to $\RP^2\times S^2$ by \cite{HKT94,HH23}. The group $\Autb(\RP^2\times S^2)$ is
therefore the natural first object in a study of $\Autb(N)$; see
\cref{sec:manifoldN} for the facts about $N$ that we use.

\section{Conventions and Stiefel--Whitney classes}
\label{sec:preliminaries}

\subsection{Pin structures}
Let $X$ be a smooth manifold. A \emph{tangential} $\mathrm{Pin}^{+}$-structure on $X$ exists if and only if $w_2(X)=0$, and a
tangential $\mathrm{Pin}^{-}$-structure exists if and only if $w_2(X)+w_1(X)^2=0$. For the stable normal bundle $\nu$ we have
$w_2(\nu)=w_2(X)+w_1(X)^2$ and $w_1(\nu)=w_1(X)$, so a \emph{normal} $\mathrm{Pin}^{+}$-structure exists if and only if $w_2(\nu)=0$, and a
normal $\mathrm{Pin}^{-}$-structure if and only if $w_2(\nu)=w_1^2$. Thus the normal labelling interchanges the two tangential ones, and
we always say which of the two is meant. In bordism computations (\cref{sec:braid}) we use the normal convention, which is also the one
of \cite{HKT94}. With this convention
$\Ompin_*=\Z/2,\Z/2,\Z/8,0,0,0$ and
$\Ompinm_*=\Z/2,0,\Z/2,\Z/2,\Z/16,0$ in degrees $0,\dots,5$; in the tangential
labelling the two lists are interchanged. For $M=\RP^2\times S^2$ one finds
$w(\nu_M)=(1+a)^{-3}=1+a$, so $w_2(\nu_M)=0$ and $M$ is normal $\mathrm{Pin}^{+}$,
equivalently tangential $\mathrm{Pin}^{-}$, in agreement with
\cref{prop:SW}.

\begin{definition}\label{def:gauss}
The \emph{stable normal Gauss map} of a smooth manifold $M^n$ is the map
$\nu_M\colon M\to\BO$ classifying the stable normal bundle.
\end{definition}

\begin{definition}\label{def:twisted}
Let $w=w_1(M)$ be the orientation character. Then $\Zw$ is the local system
with fibre $\Z$ on which $\gamma\in\pi_1(M)$ acts by $(-1)^{w(\gamma)}$. We
write $H_*(M;\Zw)$ and $H^*(M;\Zw)$ for the corresponding (co)homology.
\end{definition}

\begin{proposition}\label{prop:SW}
  Let $a\in H^1(\RP^2;\Z/2)$ and $b\in H^2(S^2;\Z/2)$ be the generators,
  pulled back to the relevant total spaces.
  \begin{enumerate}[label=(\roman*)]
    \item For $M=\RP^2\times S^2$ we have
      $H^*(M;\Z/2)=\Z/2[a,b]/(a^3,b^2)$, $w(TM)=(1+a)^3=1+a+a^2$, so
      \[
        w_1(M)=a,\qquad w_2(M)=a^2=w_1(M)^2 .
      \]
      Thus $M$ has a tangential $\mathrm{Pin}^{-}$-structure and no tangential
      $\mathrm{Pin}^{+}$-structure. The Wu class is $v_2(M)=w_2+w_1^2=0$.
    \item Let $M'=S(2L\oplus\varepsilon)\to\RP^2$ be the sphere bundle of
      $2L\oplus\varepsilon$, where $L$ is the tautological line bundle. Then
      $w(TM')=p^*\bigl((1+a)^3(1+a)^2\bigr)=p^*(1+a)^5=1+a$, so $w_1(M')=a$, $w_2(M')=0$,
      $v_2(M')=a^2\neq0$. Thus $M'$ has a tangential $\mathrm{Pin}^{+}$-structure.
  \end{enumerate}
  Hence $M$ and $M'$ are not homotopy equivalent, since the Wu class
  $v_2$ is a homotopy invariant.
\end{proposition}
\begin{proof}
  (i) $T(\RP^2)\oplus\varepsilon=3L$ and $TS^2\oplus\varepsilon$ is trivial, so
  $w(TM)=(1+a)^3=1+3a+3a^2=1+a+a^2$ because $a^3=0$.
  (ii) $TM'\oplus\varepsilon\cong p^*(T\RP^2\oplus 2L\oplus\varepsilon)$, hence
  $w(TM')=p^*((1+a)^3(1+a)^2)=p^*(1+a)^5$ and $(1+a)^5=1+5a+10a^2=1+a$ in
  $H^*(\RP^2;\Z/2)$. The last statement is Wu's formula: $v_2$ is determined by
  the action of $Sq^2$ on $H^2(-;\Z/2)$ together with Poincar\'e duality.
\end{proof}

\section{Algebraic topology of \texorpdfstring{$M=\RP^2\times S^2$}{M = RP2 x S2}}
\label{sec:twisted}

Let $M=\RP^2\times S^2$, $\pi_1(M)=G=\Z/2=\langle T\rangle$, and
$\wt M=S^2\times S^2$ with deck transformation $T(x,y)=(-x,y)$, which has
degree $-1$.

\begin{proposition}\label{prop:homotopy}
  \begin{enumerate}[label=(\roman*)]
    \item $\pi_2(M)\cong\Z^{-}\oplus\Z$ with $T_*=\operatorname{diag}(-1,1)$,
      where $\Z^-$ is $\Z$ with the sign action.
    \item $\pi_3(M)\cong\Z^2$ and $\pi_4(M)\cong(\Z/2)^2$, with trivial
      $G$-action on both.
    \item $H^n(M;\Z)=\Z,\,0,\,\Z\oplus\Z/2,\,0,\,\Z/2$ for $n=0,\dots,4$.
    \item $H_n(M;\Zw)=\Z/2,\,0,\,\Z\oplus\Z/2,\,0,\,\Z$ for $n=0,\dots,4$.
  \end{enumerate}
\end{proposition}
\begin{proof}
  (i)--(ii): $\pi_n(M)=\pi_n(S^2\times S^2)$ for $n\ge2$, and $\pi_3(S^2)=\Z$,
  $\pi_4(S^2)=\Z/2$. The deck transformation acts on $\pi_n(S^2\times S^2)$
  factorwise, by the antipodal map ($\deg -1$) on the first factor and the
  identity on the second. On $\pi_3$ and $\pi_4$ of $S^2$ the action of a
  degree $-1$ map is trivial: on $\pi_3(S^2)=\Z\langle\eta\rangle$ it
  multiplies by $(\deg)^2=1$, and on $\pi_4(S^2)=\Z/2$ every automorphism is
  trivial.
  (iii) is the K\"unneth formula, using $H^*(\RP^2;\Z)=\Z,0,\Z/2$.
  (iv) Compute $H_*(\RP^2;\Zw)$ from the cellular chain complex
  $\Z\xrightarrow{\,0\,}\Z\xrightarrow{\,2\,}\Z$ in degrees $2,1,0$ with the
  twist: $\partial_1=T-1=-2$ on the $1$-cell and $\partial_2=1+T=0$ on the
  $2$-cell, giving $H_0=\Z/2$, $H_1=0$, $H_2=\Z$. Then K\"unneth gives
  $H_*(M;\Zw)=H_*(\RP^2;\Zw)\otimes H_*(S^2)$, which is the stated list.
\end{proof}

\section{The group \texorpdfstring{$\Autb(M)$}{Aut(M)}}
\label{sec:automorphisms}

Throughout, $\Autb(M)$ denotes the group of based homotopy classes of based
homotopy self-equivalences of $M$. No orientation condition is imposed;
instead we use a sign invariant.

\subsection{The sign homomorphism}
Fix a lift $\wt x_0$ of $x_0$. A based self-equivalence $f$ has a unique lift
$\wt f$ fixing $\wt x_0$, and the other lift is $T\wt f$, of opposite degree.
Set $s(f)=\deg\wt f\in\{\pm1\}$. Uniqueness of the based lift makes $s$ a
homomorphism $\Autb(M)\to\{\pm1\}$, and we put $\Autb^{+}(M)=\ker s$.

\begin{remark}
  For a non-orientable $M$ there is no distinguished lift of $f$ to $\wt M$ preserving orientation: only one of $\wt f$ and
  $T\wt f$ has degree $+1$. The sign $s$ is defined using the based lift.
\end{remark}

\subsection{Reduction to the \texorpdfstring{$2$}{2}-type}
Let $P=P_2(M)=P_2(\RP^2)\times\CP^\infty$ be the second Postnikov section,
with $\pi_1=G$ and $\pi_2=\Z^-\oplus\Z$. The map $M\to P$ is $3$-connected.
Write $\pi\colon\Autb(M)\to\Autb(P)$.

\begin{lemma}\label{lem:AutP}
  $\Autb(P)$ has order $8$. More precisely, there is an exact sequence
  \[
    0\to H^2(G;\pi_2)\to\Autb(P)\to\Aut_G(\pi_2)\to1,
  \]
  with $H^2(G;\pi_2)\cong\Z/2$ and $\Aut_G(\pi_2)=\{\operatorname{diag}(\pm1,\pm1)\}\cong(\Z/2)^2$.
\end{lemma}
\begin{proof}
  Every self-equivalence induces the identity on $G$ since $\Aut(\Z/2)=1$.
  An equivariant endomorphism of $\Z^-\oplus\Z$ has no off-diagonal terms:
  an equivariant map $\phi\colon\Z^-\to\Z$ satisfies $\phi(-n)=\phi(n)$, hence
  $2\phi=0$, hence $\phi=0$, and similarly in the other direction. So
  $\Aut_G(\pi_2)=\{\pm1\}^2$. Based self-maps of $P$ inducing the identity on $G$ and on $\pi_2$ are classified up to homotopy by $H^2(G;\pi_2)$, acting by translation. The action is free because $H^1(K(\pi_2,2);\pi_2)=0$, and based homotopies are fibrewise over $BG$ because $\mathrm{Map}_\bullet(P,BG)$ is homotopy discrete (\cref{lem:canon}). Every $\operatorname{diag}(\pm1,\pm1)$ preserves the $k$-invariant in $H^3(G;\pi_2)=\Z/2$. Finally $H^2(\Z/2;\Z^-)=0$ and $H^2(\Z/2;\Z)=\Z/2$.
\end{proof}

\subsection{The generators}
\begin{construction}\label{con:gens}
  Let $r\colon\RP^2\to\RP^2$, $[x_1{:}x_2{:}x_3]\mapsto[-x_1{:}x_2{:}x_3]$ and
  $\rho\colon S^2\to S^2$ a reflection, both fixing the basepoints. Put
  $a=r\times\id$ and $b=\id\times\rho$. Let $\psi\colon\RP^2\to SO(3)$ be the
  composite $\RP^2\subset\RP^3=SO(3)$, based at $1$, and
  $s_\psi([x],y)=([x],\psi([x])\,y)$.
\end{construction}

\begin{lemma}\label{lem:gens}
  The maps $a,b,s_\psi$ are pairwise commuting involutions in $\Autb(M)$.
  Their images under $\pi$ generate $\Autb(P)$, and $s_\psi^*b=b+a^2$ on
  $H^2(M;\Z/2)$ whereas $a^*b=b^*b=b$.
\end{lemma}
\begin{proof}
  $a^2=b^2=\id$ and $ab=ba$ are clear. We have $s_\psi^2=s_{\psi^2}$, where
  $\psi^2=\psi\cdot\psi$ is the pointwise square. It lifts to $S^3$, since
  $\psi^2$ is trivial on $\pi_1$, and every based map $\RP^2\to S^3$ is
  null-homotopic. So $s_\psi^2\simeq\id$ (based). Next, $as_\psi a^{-1}=s_{\psi\circ r}$
  and $\psi\circ r\simeq\psi$ as based maps, since both are non-trivial on
  $\pi_1=\Z/2$ and the obstructions to a based homotopy lie in
  $H^2(\RP^2;\pi_2(\RP^3))=0$. Finally $bs_\psi b^{-1}=s_{\rho\psi\rho}$, and
  conjugation of $SO(3)$ by $\rho$ equals conjugation by $-\rho\in SO(3)$, hence
  is homotopic to the identity through based automorphisms of $SO(3)$.
  This gives pairwise commutation.

  On $\pi_2=\Z^-\oplus\Z$ the maps $a,b$ act by $\operatorname{diag}(-1,1)$ and
  $\operatorname{diag}(1,-1)$ (the based lift of $r$ is a reflection of $S^2$).
  The map $s_\psi$ acts trivially on $\pi_1,\pi_2$. To see it is not
  homotopic to the identity, compute on mod $2$ cohomology. The orbit map
  $p\colon SO(3)=\RP^3\to S^2$ is the unit tangent bundle of $S^2$ with
  Euler number $2$, so by the Gysin sequence $p^*\colon H^2(S^2;\Z)\to H^2(\RP^3;\Z)=\Z/2$ is onto,
  hence so is its mod $2$ reduction, that is $p^*b=u^2$ for $u$ the generator of
  $H^1(\RP^3;\Z/2)$. With $\mu\colon SO(3)\times S^2\to S^2$ the action,
  $\mu^*b=1\otimes b+u^2\otimes1$, and $\psi^*u=a$. Therefore
  $s_\psi^*b=b+a^2\ne b$. Since $a,b$ act trivially on $H^2(M;\Z/2)$ and
  $s_\psi$ does not, $\pi(s_\psi)$ is not in the subgroup generated by
  $\pi(a),\pi(b)$, and is nontrivial in $H^2(G;\pi_2)$ because it acts trivially on $\pi_*$,
  $*\le2$. Hence $\pi(a),\pi(b),\pi(s_\psi)$ generate a group of order $8$.
\end{proof}

\subsection{Pinch-fold maps}\label{sec:pinchfold}
Let $M^{(3)}\subset M$ be the $3$-skeleton, so $M=M^{(3)}\cup_\varphi e^4$,
and let $\mathrm{pinch}\colon M\to M\vee S^4$ collapse the boundary of a
small disc in the top cell. For $\kappa\in\pi_4(M)$ define
\[
  k_\kappa=\nabla\circ(\id_M\vee\kappa)\circ\mathrm{pinch}\colon M\to M .
\]
Since $\kappa$ is a based map from $S^4$ and $\mathrm{pinch}$ is based, $k_\kappa$ is based.

\begin{lemma}\label{lem:pinch}
  (i) $\kappa\mapsto[k_\kappa]$ is a homomorphism $\pi_4(M)\to\Autb(M)$ with
  image central. (ii) $k_\kappa$ induces the identity on $\pi_1,\pi_2$, and
  $s(k_\kappa)=1$.
\end{lemma}
\begin{proof}
  (i) Additivity: $k_\kappa k_\lambda$ and $k_{\kappa+\lambda}$ both pinch off a
  small disc, and the two-disc argument (shrink two disjoint discs in the top
  cell, then slide one across the other) identifies them; the same argument
  gives commutativity of $k_\kappa$ and $k_\lambda$. It suffices to treat $g\in\{a,b,s_\psi\}$, which generate $\Autb(M)$ together with the image of $\pi_4(M)$, an abelian group. These are diffeomorphisms, so after an isotopy fixing the base point they preserve the small disc, and then
  $g\circ k_\kappa\simeq k_{\varepsilon g_*\kappa}\circ g$ with $\varepsilon=\pm1$, where we use that $\pi_1(M)$ acts trivially on $\pi_4(M)$ (\cref{prop:homotopy}) to change the arc joining the disc to the base point. Since $2\kappa=0$, it is enough that $g_*$ acts trivially on
  $\pi_4(M)=\pi_4(S^2)^2$. For $a$ and $b$ this holds since
  they are products of maps of $\RP^2$ and $S^2$ and a self-map of $S^2$ acts
  trivially on $\pi_4(S^2)=\Z/2$. For $s_\psi$, the lift $(x,y)\mapsto(x,\wt\psi(x)y)$ restricts
  on $S^2\times\{y_0\}$ to $x\mapsto(x,\wt\psi(x)y_0)$, whose second component
  has degree $0$ (it factors through $\RP^2$), hence is null-homotopic, and it is
  the identity on $\{x_0\}\times S^2$; so $s_\psi$ is the identity on
  $\pi_4(S^2)\oplus\pi_4(S^2)$. Hence $k_\kappa$ is central. (ii) $k_\kappa$ is the identity on $M^{(3)}$, which contains the
  $2$-skeleton. The based lift $\wt k_\kappa$ agrees with the identity outside the two lifts of the small disc and is modified
  there by $\wt\kappa$ and $T_*\wt\kappa$. Both have Hurewicz image $0$ in $H_4(\wt M)$, so $\deg\wt k_\kappa=1$.
  Hence $k_\kappa$ is a homotopy equivalence by \cref{lem:PDup}.
\end{proof}

\begin{proposition}\label{prop:kernel}
  $\ker\pi=\{k_\kappa:\kappa\in\pi_4(M)\}$ and $\kappa\mapsto k_\kappa$ is injective.
  Hence $\ker\pi\cong\pi_4(M)\cong(\Z/2)^2$.
\end{proposition}
\begin{proof}
  \emph{Surjectivity onto $\ker\pi$.} Let $f$ be a based self-equivalence
  with $\pi(f)\simeq\id_P$. Compare $f|_{M^{(3)}}$ with the inclusion. The
  obstructions to a based homotopy lie in $H^i(M;\pi_i(M))$ for $i\le3$.
  The homotopy fibre of $M\to P$ is $2$-connected, and $\pi(f)\simeq\id_P$ gives a homotopy $cf\simeq c$. It lifts over $M^{(3)}\times I$, because the obstruction lies in
  $H^3(M;\pi_3(M))=H^3(M;\Z)^2=0$ by \cref{prop:homotopy}(iii). So $f\simeq\id$ on
  $M^{(3)}$, and $f$ differs from $\id$ by a modification on the top cell, that
  is by an element of $\pi_4(M)$, so $f\simeq k_\kappa$.

  \emph{Injectivity.} Write $\kappa=\kappa^{(1)}+\kappa^{(2)}$ with $\kappa^{(1)}\in\pi_4(\RP^2)$ and
  $\kappa^{(2)}\in\pi_4(S^2)$, both $\cong\Z/2$. A map into a product is determined up to homotopy by its two
  components, and
  \[
    \mathrm{pr}_1\circ k_\kappa=\mathrm{pr}_1+\kappa^{(1)},\qquad \mathrm{pr}_2\circ k_\kappa=\mathrm{pr}_2+\kappa^{(2)},
  \]
  where $+$ denotes the action of $\pi_4$ on maps out of $M$ by modification on the top cell. If
  $k_\kappa\simeq\id$ then both components are homotopic to $\mathrm{pr}_1,\mathrm{pr}_2$, even freely. So it is enough to show:
  (A) $\mathrm{pr}_1+\kappa^{(1)}\not\simeq\mathrm{pr}_1$ freely if $\kappa^{(1)}\neq0$, and
  (B) $\mathrm{pr}_2+\kappa^{(2)}\not\simeq\mathrm{pr}_2$ freely if $\kappa^{(2)}\ne0$.
  Free homotopy classes of maps $\RP^2\times S^2\to X$ correspond to free homotopy classes
  of maps $S^2\to\mathrm{Map}(\RP^2,X)$, and also to those of maps $\RP^2\to\mathrm{Map}(S^2,X)$.

  (A) Let $F=\mathrm{Map}(\RP^2,\RP^2)_{\id}$ and $G=\mathrm{Map}_*(\RP^2,\RP^2)_{\id}$, the components of the
  identity. The adjoint of $\mathrm{pr}_1$ is the constant map $S^2\to F$ at $\id$, and the adjoint of
  $\mathrm{pr}_1+\kappa^{(1)}$ is the class $P(\kappa^{(1)})\in\pi_2(F,\id)$, where
  $P\colon\pi_2(\Omega^2\RP^2)=\pi_4(\RP^2)\to\pi_2(G)\to\pi_2(F)$ is induced by the inclusion of the
  fibre of the restriction $G\to\Omega_c\RP^2$ to $\RP^1$ (maps of the $2$-cell relative to its boundary), followed by the inclusion of $G$ as the fibre
  of the evaluation $\mathrm{ev}\colon F\to\RP^2$. A free homotopy class of a map from a sphere is trivial if and only if the corresponding
  element of $\pi_2$ is trivial, so we must show $P$ is injective.
  First, $\pi_2(\Omega^2\RP^2)\to\pi_2(G)$ is injective, because the preceding boundary map
  $\pi_3(\Omega_c\RP^2)=\pi_4(\RP^2)\to\pi_2(\Omega^2\RP^2)=\pi_4(\RP^2)$ is $x\mapsto(1+c_*)x=2x=0$: the boundary is
  induced by the degree $2$ attaching map of the $2$-cell, which on $\Omega_c$ sends a loop $\gamma$ to $\gamma\cdot\gamma$, and $\pi_4(\RP^2)=\Z/2$.
  Second, $\pi_2(G)\to\pi_2(F)$ has kernel the image of $\partial\colon\pi_3(\RP^2)\to\pi_2(G)$, the boundary of the evaluation
  fibration. The group $\pi_3(\RP^2)\cong\pi_3(S^2)\cong\Z$ is generated by the image of $\pi_3(SO(3))$ under the orbit map
  $SO(3)\to\RP^2=SO(3)/O(2)$, which is an isomorphism on $\pi_3$ since $\pi_3(O(2))=\pi_2(O(2))=0$. Since $SO(3)$ acts on $\RP^2$ by isometries,
  the orbit map factors as $SO(3)\to F\xrightarrow{\ \mathrm{ev}\ }\RP^2$, so every element of $\pi_3(\RP^2)$ lifts to $\pi_3(F)$ and $\partial=0$. Hence $P$ is injective.

  (B) The adjoint of $\mathrm{pr}_2$ is the constant map $\RP^2\to\mathrm{Map}_1(S^2,S^2)$ at $\id$. The adjoint of
  $\mathrm{pr}_2+\kappa^{(2)}$ is $\mathrm{pinch}^*(\kappa^{(2)})$, where $\mathrm{pinch}\colon\RP^2\to S^2$ and $\kappa^{(2)}$ is viewed in
  $\pi_2(\Omega^2_1S^2)\to\pi_2(\mathrm{Map}_1(S^2,S^2))$. Consider the fibre sequence
  $\Omega S^2\xrightarrow{\ \partial\ }\Omega^2_1S^2\to\mathrm{Map}_1(S^2,S^2)\to S^2$. On homotopy groups $\partial(x)=\pm[x,\iota_2]$. Since
  $\RP^2=S^1\cup_2e^2$, the cofibre sequences give
  \[
\begin{gathered}
    [\RP^2,\Omega S^2]\cong[\Sigma\RP^2,S^3]\cong\Z/2,\\ [\RP^2,\Omega^2S^2]\cong[\Sigma^2\RP^2,S^2]\cong\pi_4(S^2)/2\pi_4(S^2)=\Z/2,
  \end{gathered}
\]
  where in the first group we used $\Omega S^2\simeq S^1\times\Omega S^3$ and $[\RP^2,S^1]=H^1(\RP^2;\Z)=0$. The generator of the first group is
  $\mathrm{pinch}^*\tilde\eta$ with $\tilde\eta\in\pi_2(\Omega S^2)=\pi_3(S^2)$ the Hopf class, so $\partial_*$ of it is
  $\mathrm{pinch}^*[\eta_2,\iota_2]=0$, since $\eta_2$ is the image of the generator of $\pi_3(SU(2))$ under the orbit map of the $SU(2)$-action on $S^2$, hence a Gottlieb
  element. Therefore $\partial_*=0$, and by exactness $[\RP^2,\Omega^2_1S^2]\to[\RP^2,\mathrm{Map}_1(S^2,S^2)]$ has trivial kernel. The class
  $\mathrm{pinch}^*\kappa^{(2)}$ is the generator of $[\RP^2,\Omega^2S^2]=\Z/2$, hence it is non-trivial in $[\RP^2,\mathrm{Map}_1(S^2,S^2)]$,
  where free and based classes agree for the trivial class.
\end{proof}

\begin{remark}\label{rem:caveat}
  The pairing of $\kappa^{(1)}$ with the $\RP^2$-factor is where non-orientability enters: the lift of $k_{\kappa^{(1)}}$ to the universal cover is
  $\id+2\wt\kappa^{(1)}\simeq\id$, and the normal invariant of $k_{\kappa^{(1)}}$ is expected to vanish (\cref{rem:k1}), so neither the cover nor, presumably,
  surgery theory sees it.
\end{remark}

\subsection{Proof of \texorpdfstring{\cref{thm:main}}{Theorem 1.1}}
By \cref{prop:kernel} there is an exact sequence
\[
  0\to\pi_4(M)\to\Autb(M)\xrightarrow{\ \pi\ }\Autb(P).
\]
By \cref{lem:gens}, $\pi$ is onto $\Autb(P)$, a group of order $8$ (\cref{lem:AutP}).
Hence $|\Autb(M)|=4\cdot8=32$. The elements $a,b,s_\psi$ generate a subgroup
$S\cong(\Z/2)^3$ mapping isomorphically onto $\Autb(P)$ (\cref{lem:gens}),
and the kernel $K=\pi_4(M)\cong(\Z/2)^2$ is central (\cref{lem:pinch}). Therefore
$\Autb(M)=S\times K\cong(\Z/2)^5$. The sign $s$ is $-1$ on $a$ and $b$ (their
based lifts are $\wt r\times\id$ and $\id\times\wt\rho$), and $+1$ on $s_\psi$
and $k_\kappa$. The lift of $s_\psi$ is $(x,y)\mapsto(x,\tilde\psi(x)y)$, which
has degree $1$. Hence $\ker s=\langle ab,s_\psi,k_1,k_2\rangle\cong(\Z/2)^4$. \qed

\subsection{Free homotopy classes}
\begin{proof}[Proof of \cref{cor:free}]
  Let $P_\ell$ be the homeomorphism obtained by pushing the basepoint once around
  the generator of $\pi_1(M)$ along an isotopy. It is based, freely isotopic to
  the identity, and its based lift is homotopic to $T$, so $s(P_\ell)=-1$. The kernel of the
  surjection $\Autb(M)\to\Autf(M)$ is the image of the connecting map $\pi_1(M,x_0)\to\pi_0\mathrm{Map}_*(M,M)$ of the evaluation fibration. It has order at most $2$ and is generated by $[P_\ell]$, which is non-trivial because $s(P_\ell)=-1$. Thus $\Autf(M)\cong\Autb(M)/\langle P_\ell\rangle$, and since
  $s(P_\ell)=-1$, the subgroup $\Autb^{+}(M)$ maps isomorphically onto it.
\end{proof}

\section{A non-orientable Hambleton--Kreck braid in dimension \texorpdfstring{$4$}{4}}
\label{sec:braid}

The braid of \cite{HK04} is stated there for closed oriented $4$-manifolds
(spin case Theorem~2.16, non-spin case Theorem~3.15; the description of
$\Autb(M)$ for odd order $\pi_1$ is Theorem~A with Corollary~4.9). In this section we construct a version
for closed smooth non-orientable $4$-manifolds with finite fundamental group and
spin universal cover. The proofs follow \cite[Section~2]{HK04}. The points that differ
in the non-orientable case are collected in \cref{sec:changes} and the proofs are given in \cref{sec:proofs}.

\subsection{Setup}
Let $M$ be a closed connected smooth $4$-manifold, $\pi=\pi_1(M)$ finite,
$w=w_1(M)\ne0$, and assume $w_2(\wt M)=0$ for the universal cover. Let
$c\colon M\to B\pi$ classify the universal cover. A \emph{normal $1$-type} in
the sense of Kreck is then
\[
  B=B\mathrm{Spin}\times B\pi\xrightarrow{\ \xi=\gamma\oplus E\ }BO,
\]
where $E\to B\pi$ is a stable bundle with $c^*w_1(E)=w$ and $c^*w_2(E)=w_2(\nu_M)$.
Such an $E$ must exist if $M$ has this normal $1$-type; it is a condition on $\pi$
and $\nu_M$, and for $\pi=\Z/2$ one takes $E=L$ or $E=3L$ according to whether
$w_2(\nu_M)=0$ or $w_2(\nu_M)=w^2$. For a space $X$ over $B\pi$ write
\[
  \Omega^{B}_n(X)=\Omega^{\mathrm{spin}}_n(X;E)=\pi_n\bigl(M\mathrm{Spin}\wedge X^{E-\dim E}\bigr)
\]
for the twisted spin bordism of $X$ (a $B$-structure on a manifold $N^n$ with
a map to $X$ is a spin structure on $\nu_N\ominus g^*E$). For $X=M$ and $X=P=P_2(M)$ the
bundle is pulled back along $c$.

Finiteness of $\pi$ is used only in a few places, which are listed in \cref{rem:finite}.

\begin{example}\label{ex:RP2S2B}
For $M=\RP^2\times S^2$, $w(\nu_M)=1+a$, so $E=L$ and
$\Omega^B_*(\mathrm{pt}\to B\Z/2)=\pi_*(M\mathrm{Spin}\wedge(B\Z/2)^{L-1})$ is the
normal $\mathrm{Pin}^{+}$ bordism of \cref{sec:preliminaries}:
$\Z/2,\Z/2,\Z/8,0,0,0$ in degrees $0,\dots,5$.
\end{example}

\subsection{What changes}\label{sec:changes}
Four things differ from the oriented braid.

\medskip\noindent\emph{(1) Degree.} Orientation classes live in
$H_4(M;\Zw)\cong\Z$. For a based self-equivalence $f$ the canonical isomorphism
$f^*\Zw\cong\Zw$ (it is the identity on the fibre at the basepoint) gives a
twisted degree $\deg_w(f)\in\{\pm1\}$.

\begin{lemma}\label{lem:degw}
$\deg_w(f)=s(f)$, the degree of the based lift $\wt f$ to the universal cover.
\end{lemma}
\begin{proof}
Let $p\colon\wt M\to M$ be the universal cover. The transfer
$\mathrm{tr}\colon H_4(M;\Zw)\to H_4(\wt M;\Z)$ sends $[M]_w$ to
$\sum_{g\in\pi}(-1)^{w(g)}g_*[\wt M]$. Since $g_*[\wt M]=(-1)^{w(g)}[\wt M]$, the signs cancel
and $\mathrm{tr}[M]_w=|\pi|\,[\wt M]$. If $f_*=\id$ on $\pi_1$ the based lift $\wt f$ is
$\pi$-equivariant, so the square formed by $\wt f$ and $f$ is a pullback of coverings and
$\mathrm{tr}\circ f_*=\wt f_*\circ\mathrm{tr}$. Hence $\deg_w(f)\,|\pi|=\deg\wt f\,|\pi|$.
\end{proof}

Put $\Aut^{+}_\bullet(M)=\ker s$. The group $\widehat\Omega^B_n$ below is the
kernel of the twisted degree $\Omega^B_n(M)\to H_n(M;\Zw)$ when $n=4$ and the
whole group otherwise.

\medskip\noindent\emph{(2) The Poincar\'e duality upgrade.} The oriented argument
uses that a degree one self-map inducing an isomorphism on $H_*$ is an
equivalence. For finite $\pi$ there is a cleaner route.

\begin{lemma}\label{lem:PDup}
Let $M$ be a closed connected $4$-manifold with finite $\pi_1$ and let
$f\colon M\to M$ be a based map with $f_*=\id$ on $\pi_1$ and $\deg_w(f)=1$. Then $f$ is a
homotopy equivalence.
\end{lemma}
\begin{proof}
The based lift $\wt f\colon\wt M\to\wt M$ of $f$ to the universal cover has degree
$1$ by \cref{lem:degw}. A degree one self-map of a closed oriented manifold is
split surjective on $H_*(-;\Z)$ (cap product with the fundamental class), hence
bijective, since the groups are finitely generated and therefore Hopfian. As $\wt M$ is
simply connected, $\wt f$ is a homotopy equivalence by Whitehead's theorem. Then
$f$ is an isomorphism on $\pi_1$ and on all $\pi_i$, $i\ge2$, so it is a homotopy
equivalence.
\end{proof}

{\hbadness=2000
\medskip\noindent\emph{(3) Local coefficients.} The lifting obstructions in the definition of $\gamma$ lie in $H^{i+1}(V,A;\pi_i(X))$, where $X$ is the homotopy fibre of $c$ and $\pi_i(X)$ is a $\Z\pi$-module on which $\pi$ may act non-trivially. We use cohomology with local coefficients, so no assumption on this action is needed. Poincar\'e duality for the non-orientable manifold $V$ gives
\[H^{k}(V,A;G)\cong H_{d-k}(V,C;G\otimes\Zw).\]
The vanishing argument is unchanged: $X$ is $2$-connected, and after surgery below the middle
dimension the pair $(V,C)$ is $\ell_d$-connected, where $\ell_5=1$ and $\ell_6=2$, so the relevant groups are zero. The statements are \cref{lem:surgery,lem:lifting}.
}

\medskip\noindent\emph{(4) The spin structure torsor.} Lifts of $\nu_M$ to $B$ with prescribed
$B\pi$-component and fixed orientation of $\nu_M\ominus c^*E$ form a torsor over $H^1(M;\Z/2)$, which is non-zero because $w_1(M)\ne0$, so $(M,\bar\nu)$ and $(M,\bar\nu+x)$
are different $B$-manifolds, with possibly different bordism classes. For a fixed $\bar\nu$ this causes no
difficulty in the braid (\cref{sec:transport}); it matters when one asks which self-equivalences are
realized by diffeomorphisms (\cref{rem:sminus}). The hypothesis that $\pi$ has odd order in \cite{HK04} is used there in
the computations, not for the structure of the braid.

\subsection{The braid}\label{sec:braidsub}
Fix a normal $1$-smoothing $\bar\nu=(M,\sigma)$, that is a $B$-structure on $\nu_M$, which will be
kept fixed throughout. Write $P=P_2(M)$, $c\colon M\to P$ for the $2$-type and its $3$-connected
map, and $c_X\colon X\to B\pi$ for the classifying map of the universal cover of $X=M,P$. Let
$\Aut_\bullet(X)$ denote the based homotopy classes of self-equivalences inducing the identity on
$\pi$ (for $\pi=\Z/2$ this is no restriction), and
$\Aut^{+}_\bullet(M)=\ker(s)$. Put $\widehat\Omega^B_4(M)=\ker\bigl(\Omega^B_4(M)\to H_4(M;\Zw)\bigr)$
and let $\widehat\Omega^B_5(P,M)$ be the group of bordism classes of $B$-manifolds
$(W^5,\partial W)\to(P,M)$ whose boundary map to $M$ has twisted degree $0$.

\subsubsection{Step (c): transport of structures}\label{sec:transport}
In the oriented case an element of $\Omega^{\mathrm{spin}}_n(X)$ is a manifold with a map to $X$ and
nothing else is needed to push it forward along a self-map of $X$. Here the structure is a spin
structure on $\nu_N\ominus g^*c_X^*E$, so pushing forward along $u\colon X\to Y$ requires an identification of
$c_X^*E$ with $u^*c_Y^*E$. This identification exists and is canonical.

Call a based map $u\colon X\to Y$ of connected based CW complexes with maps $c_X,c_Y$ to $B\pi$
\emph{over $B\pi$} if $c_Yu\simeq c_X$ by a based homotopy.

\begin{lemma}\label{lem:canon}
For a connected based CW complex $X$ the space $\mathrm{Map}_\bullet(X,B\pi)$ is homotopy discrete.
Consequently, for based maps $\varphi,\psi\colon X\to B\pi$ the based homotopies from $\varphi$ to $\psi$
form either the empty set or a contractible space.
\end{lemma}
\begin{proof}
Evaluation at the base point gives a fibration $\mathrm{Map}_\bullet(X,B\pi)\to\mathrm{Map}(X,B\pi)\to B\pi$.
Since $B\pi$ is aspherical, the component $\mathrm{Map}(X,B\pi)_\varphi$ is a $K(Z,1)$ where $Z$ is the
centralizer in $\pi$ of $\varphi_*\pi_1(X)$, and evaluation induces the inclusion $Z\to\pi$ on $\pi_1$.
This is injective and $\pi_k(B\pi)=0$ for $k\ge2$, so the long exact sequence gives
$\pi_k(\mathrm{Map}_\bullet(X,B\pi),\varphi)=0$ for all $k\ge1$.
\end{proof}

Let $u$ be over $B\pi$ and $H\colon c_Yu\simeq c_X$ a based homotopy. Since pull-backs of a bundle along
homotopic maps are canonically isomorphic up to homotopy, $H$ gives an isomorphism
$\Phi_H\colon c_X^*E\cong u^*c_Y^*E$ of bundles over $X$. For a $B$-manifold $(N,g,\sigma)$ over $X$,
where $\sigma$ is a spin structure on $\nu_N\ominus g^*c_X^*E$, define
\[
 u_*[N,g,\sigma]=[N,\,u\circ g,\,(g^*\Phi_H)_\#\sigma],
\]
where $(g^*\Phi_H)_\#$ transports spin structures along the induced isomorphism of stable bundles
$\nu_N\ominus g^*c_X^*E\cong\nu_N\ominus(ug)^*c_Y^*E$.

\begin{lemma}\label{lem:transport}
The homomorphism $u_*\colon\Omega^B_n(X)\to\Omega^B_n(Y)$ does not depend on $H$, only on the based
homotopy class of $u$, and $(vu)_*=v_*u_*$, $(\id)_*=\id$.
\end{lemma}
\begin{proof}
By \cref{lem:canon} any two choices of $H$ are homotopic through based homotopies, hence $\Phi_H$ is
unique up to homotopy and the resulting structures on $N$ are isomorphic; applying this to
$N\times I$ shows that $u_*$ is compatible with bordism. If $G$ is a based homotopy
$u\simeq u'$ and $H$ is as above, then $H'=H\cdot(c_Y G)^{-1}$ is a based homotopy
$c_Yu'\simeq c_X$, and the map $N\times I\to Y$ given by $G\circ(g\times\id)$ with the structure
transported along $H$ is a bordism from $u_*[N,g,\sigma]$ to $u'_*[N,g,\sigma]$. Finally, for $v\colon Y\to Z$ over $B\pi$
with homotopy $H_v$, the concatenation $c_Zvu\simeq c_Yu\simeq c_X$ is a based homotopy
and therefore \emph{is} the canonical one by \cref{lem:canon}. Thus the transport along $vu$ equals the
composite of the transports along $u$ and $v$.
\end{proof}

When $B\pi$ is a point (the oriented simply connected case) all of this is vacuous, which is why it does not
occur in \cite{HK04}.
Every $f\in\Aut_\bullet(M)$ and every $\varphi\in\Aut_\bullet(P)$ is over $B\pi$ for the classifying maps, so
\cref{lem:transport} gives actions
$f_*$ on $\Omega^B_n(M)$ and $\varphi_*$ on $\Omega^B_n(P)$, and $c_*$ intertwines $f$ with the
map $\Aut_\bullet(M)\to\Aut_\bullet(P)$, $f\mapsto\varphi_f$, where $\varphi_fc\simeq cf$. For
$f\in\Aut^+_\bullet(M)$ the induced map on $H_4(M;\Zw)$ is multiplication by $\deg_w(f)=1$
(\cref{lem:degw}), so $f_*$ preserves $\widehat\Omega^B_4(M)$.

\begin{definition}\label{def:alpha}
For $f\in\Aut^+_\bullet(M)$ put $\alpha(f)=f_*[M,\mathrm{id},\sigma]-[M,\mathrm{id},\sigma]\in\widehat\Omega^B_4(M)$.
The structure $\sigma$ lives on the source and does not depend on $f$.
\end{definition}

\begin{proposition}\label{prop:crossed}
For $f,g\in\Aut^+_\bullet(M)$ one has $\alpha(fg)=\alpha(f)+f_*\alpha(g)$. Consequently
$\ker\alpha$ is a subgroup, and $\alpha(f)=\alpha(g)$ if and only if $g^{-1}f\in\ker\alpha$.
\end{proposition}
\begin{proof}
By \cref{lem:transport}, $(fg)_*[M,\id,\sigma]=f_*g_*[M,\id,\sigma]$. Hence
\[
 \alpha(fg)=f_*\bigl(g_*[M,\id,\sigma]-[M,\id,\sigma]\bigr)+f_*[M,\id,\sigma]-[M,\id,\sigma]=f_*\alpha(g)+\alpha(f).
\]
In particular $\alpha(\id)=0$ and $\alpha(g^{-1})=-g^{-1}_*\alpha(g)$, so $\alpha(g^{-1}f)=g^{-1}_*(\alpha(f)-\alpha(g))$.
Since $g^{-1}_*$ is an automorphism, the last assertion follows, and $\ker\alpha$ is closed under products and inverses.
\end{proof}

The proof is formal once \cref{lem:transport} is available. No gauge data are needed: for a fixed $\bar\nu$ the structure is carried by the
source, and the identification $c^*E\cong f^*c^*E$ is forced by \cref{lem:canon}. The braid concerns $\Aut^+_\bullet(M)$ and not $\Aut_\bullet(M)$.

\subsubsection{The groups \texorpdfstring{$\widetilde{\mathcal H}(M)$ and $\widetilde{\mathcal H}(P)$}{H(M) and H(P)}}
\begin{definition}\label{def:H}
An element of $\widetilde{\mathcal H}(M)$ is a bordism class of pairs $(W,F)$, where $W^5$ is a compact manifold with a
$B$-structure and $\partial W=-M\sqcup M$, the $B$-structures on the two ends being $\sigma$
(under $\id_M$), and $F\colon W\to M$ is a map over $B\pi$ with $F|_{\partial_1W}=\id_M$ and
$F|_{\partial_2W}=f$ for a based self-map $f$. Two pairs $(W,F)$, $(W',F')$ with end maps $f,f'$ are
identified if there is a based homotopy $f\simeq f'$ such that the closed manifold
$-W'\cup W\cup M\times I$ with the induced map to $M$ is zero in $\Omega^B_5(M)$.
\end{definition}

The end map $f$ is a self-map over $B\pi$ with $F_*[W,\partial]=0$ in homology, so $\deg_w(f)=1$, and $f$ is a
homotopy equivalence by \cref{lem:PDup}. Hence $f\in\Aut^+_\bullet(M)$. Define the product by
\[
 (W,F)\bullet(W',F')=\bigl(W\cup_{\partial_2W=\partial_1W'}W',\ F\cup f\circ F'\bigr).
\]
Here $W'$ is regarded as a $B$-manifold over $M$ through $f\circ F'$ with the structure transported by $f_*$
(\cref{lem:transport}). The glueing is legitimate: the incoming end of $W'$ becomes $(M,f,\sigma)$,
which is the outgoing end of $W$.

\begin{proposition}\label{prop:Hgroup}
$\bullet$ is well defined on equivalence classes, associative, has unit $(M\times I,\mathrm{pr}_1)$ and inverses.
The map $\widetilde{\mathcal H}(M)\to\Aut^+_\bullet(M)$, $(W,F)\mapsto f$, is a homomorphism with image $\ker\alpha$, and the
map $\Omega^B_5(M)\to\widetilde{\mathcal H}(M)$, $[N,g]\mapsto(M\times I\sqcup N,\mathrm{pr}_1\sqcup g)$,
is a homomorphism whose image is the kernel of this map.
\end{proposition}
\begin{proof}
\emph{Well defined.} For $(W_1,F_1)$, $(V_1,G_1)$ with end maps $f,f'$, a homotopy $f\simeq f'$ and
closed manifold $Z_1\in\Omega^B_5(M)$, and a second element $(W_2,F_2)$, the closed manifold
associated with the products $(W_1,F_1)\bullet(W_2,F_2)$ and $(V_1,G_1)\bullet(W_2,F_2)$ is $Z_1$
together with two copies of $W_2$, mapped by $f\circ F_2$ and $f'\circ F_2$ and joined by the homotopy composed with $F_2$.
The two copies of $W_2$ bound $W_2\times I$, so the class is $Z_1$. In the other variable, with $Z_2$ the closed
manifold for $(W_2,F_2)\sim(V_2,G_2)$, the closed manifold for the products is $f_*Z_2$ by the definition of the product and
\cref{lem:transport}, and it is zero in $\Omega^B_5(M)$ because $Z_2$ is. In both variables the glueing maps
and structures agree by \cref{lem:canon}, since all identifications of $c^*E$ with $f^*c^*E$ are forced.

\emph{Associativity and unit.} Associativity of glueing, together with $(ff')_*=f_*f'_*$ from \cref{lem:transport}.

\emph{Inverses.} Given $(W,F)$ with end map $f$, let $\overline W=-W$ be $W$ with the $B$-structure reversed, read as a cobordism from $\partial_2W$ to $\partial_1W$
and give it the map $f^{-1}\circ F$, where $f^{-1}\in\Aut^+_\bullet(M)$ is a homotopy inverse. Its end maps are
$f^{-1}f\simeq\id$ and $f^{-1}$, and the product $(W,F)\bullet(\overline W,f^{-1}F)$ is $W\cup\overline W$ with the
map $F\cup F$ (using $ff^{-1}\simeq\id$), which is the boundary of $W\times I$ glued to $M\times I$. Hence it represents the unit.

\emph{Homomorphisms and image.} The end map is multiplicative by definition of $\bullet$. A class $f\in\Aut^+_\bullet(M)$
is in the image iff $(M,f,\sigma)$ and $(M,\id,\sigma)$ are bordant over $M$ as $B$-manifolds, that is iff $f_*[M,\id,\sigma]=[M,\id,\sigma]$,
that is iff $\alpha(f)=0$. The map from $\Omega^B_5(M)$ is a homomorphism because $\bullet$ with a class supported away from the cylinder
is disjoint union, and an element $(W,F)$ with $f\simeq\id$ gives, by glueing $M\times I$ along the homotopy, a closed manifold whose class in
$\Omega^B_5(M)$ maps to $(W,F)$.
\end{proof}

The group $\widetilde{\mathcal H}(P)$ is defined in the same way with $P$ in place of $M$, the end map
$f\colon M\to P$ being a based $3$-connected map over $B\pi$ and $F|_{\partial_1W}=c$; the product uses the
unique $\varphi_f\in\Aut_\bullet(P)$ with $\varphi_fc\simeq f$, and $(\varphi_f)_*$ from \cref{lem:transport}.

\subsubsection{The theorem}
The maps $\pi$, $\beta$, $\partial$ and $\gamma$ in the theorem are defined in \cref{sec:proofs}.
\begin{theorem}\label{thm:braidNO}
Let $M$ be a closed smooth non-orientable $4$-manifold with finite $\pi_1$, spin universal cover
and normal $1$-smoothing $\bar\nu$. There are exact sequences
\begin{align*}
&\Omega^B_5(M)\to\widetilde{\mathcal H}(M)\to\Aut^+_\bullet(M)\xrightarrow{\ \alpha\ }\widehat\Omega^B_4(M),\\
&\pi_1(\mathcal E_\bullet(P))\to\Omega^B_5(P)\to\widetilde{\mathcal H}(M)\to\Aut_\bullet(P)\xrightarrow{\ \beta\ }\Omega^B_4(P),\\
&\Omega^B_5(M)\to\Omega^B_5(P)\to\widehat\Omega^B_5(P,M)\xrightarrow{\ \partial\ }\widehat\Omega^B_4(M)\to\Omega^B_4(P),\\
&\pi_1(\mathcal E_\bullet(P))\to\widehat\Omega^B_5(P,M)\xrightarrow{\ \gamma\ }\Aut^+_\bullet(M)\xrightarrow{\ \pi\ }\Aut_\bullet(P),
\end{align*}
which interlock as a braid: $\alpha\circ\gamma=\partial$, and the composite
$\Omega^B_5(P)\to\widetilde{\mathcal H}(M)\to\Aut^+_\bullet(M)$ equals
$\gamma$ on the image of $\Omega^B_5(P)$ up to inversion. Here $\alpha$ is a crossed homomorphism
(\cref{prop:crossed}) and exactness at $\Aut^+_\bullet(M)$ is exactness of pointed sets in the sense that
$\alpha(f)=\alpha(g)$ iff $g^{-1}f\in\ker\alpha=\im(\widetilde{\mathcal H}(M))$.
\end{theorem}

\begin{figure}[ht]
\begin{center}
\begin{tikzpicture}[x=2.55cm,y=1.55cm,>=stealth,every node/.style={font=\small},lab/.style={font=\scriptsize,inner sep=1pt}]
\node (A1) at (0,2) {$\Omega^B_5(M)$};
\node (A2) at (2,2) {$\widetilde{\mathcal H}(M)$};
\node (A3) at (4,2) {$\Aut_\bullet(P)$};
\node (B1) at (1,1) {$\Omega^B_5(P)$};
\node (B2) at (3,1) {$\Aut^+_\bullet(M)$};
\node (B3) at (5,1) {$\Omega^B_4(P)$};
\node (C1) at (0,0) {$\pi_1(\mathcal E_\bullet(P))$};
\node (C2) at (2,0) {$\widehat\Omega^B_5(P,M)$};
\node (C3) at (4,0) {$\widehat\Omega^B_4(M)$};
\foreach \a/\b in {A1/B1,A2/B2,A3/B3,B1/C2,B2/C3}
  \draw[->,shorten >=3pt,shorten <=3pt] (\a) -- (\b);
\foreach \a/\b in {B1/A2,B2/A3,C1/B1,C2/B2,C3/B3}
  \draw[->,shorten >=3pt,shorten <=3pt] (\a) -- (\b);
\draw[->,shorten >=3pt,shorten <=3pt] (A1) to[bend left=28] (A2);
\draw[->,shorten >=3pt,shorten <=3pt] (A2) to[bend left=28] (A3);
\draw[->,shorten >=3pt,shorten <=3pt] (C1) to[bend right=28] (C2);
\draw[->,shorten >=3pt,shorten <=3pt] (C2) to[bend right=28] node[lab,below]{$\partial$} (C3);
\node[lab] at (4.62,1.58) {$\beta$};
\node[lab] at (3.36,1.58) {$\pi$};
\node[lab] at (3.62,0.45) {$\alpha$};
\node[lab] at (2.62,0.45) {$\gamma$};
\end{tikzpicture}
\end{center}
\caption{The braid of \cref{thm:braidNO}. Each of the four sequences of the theorem is read along one of the two rising or falling diagonals together with one arc; for example the first sequence is $\Omega^B_5(M)\to\widetilde{\mathcal H}(M)\to\Aut^+_\bullet(M)\xrightarrow{\alpha}\widehat\Omega^B_4(M)$, and the third is $\Omega^B_5(M)\to\Omega^B_5(P)\to\widehat\Omega^B_5(P,M)\xrightarrow{\partial}\widehat\Omega^B_4(M)\to\Omega^B_4(P)$. Commutativity is $\alpha\gamma=\partial$.}
\label{fig:braid}
\end{figure}

\subsection{Proof of \texorpdfstring{\cref{thm:braidNO}}{the braid theorem}}\label{sec:proofs}

Throughout this subsection $M$, $\pi$, $E$, $B$, $\bar\nu=(M,\sigma)$, $P=P_2(M)$ and $c$ are as in \cref{sec:braidsub}. We replace $P$ by the mapping cylinder of $c$, so that $M\subset P$ and $c$ is the inclusion. The map $c$ is an isomorphism on $\pi_i$ for $i\le2$ and $P$ has no homotopy above dimension $2$, so $c$ is $3$-connected, and the homotopy fibre $X$ of $c$ is $2$-connected. A \emph{$B$-manifold over $P$} is a manifold $V$ with a map $\Phi\colon V\to P$ and a spin structure on $\nu_V\ominus\Phi^*c_P^*E$, and a $B$-manifold over $M$ is defined in the same way. All identifications of the bundles $\Phi^*c_P^*E$ that are used in glueing and in comparing structures are the canonical ones of \cref{lem:canon}, and we do not mention them again. If $V$ is a $B$-manifold over $P$, then $w_1(V)=\Phi^*w$, so the orientation local system of $V$ is $\Phi^*\Zw$.

For a closed $4$-dimensional $B$-manifold $(N,g,\sigma')$ over $M$ the spin structure orients the virtual bundle $\nu_N\ominus g^*c^*E$ and so determines a fundamental class $[N]_{\sigma'}\in H_4(N;g^*\Zw)$. Its \emph{twisted degree} is the integer $d$ with $g_*[N]_{\sigma'}=d\,[M]_\sigma$ in $H_4(M;\Zw)\cong\Z$. It is a bordism invariant, and $\widehat\Omega^B_4(M)$ is its kernel.

\subsubsection{The maps \texorpdfstring{$\pi$, $\beta$ and $\partial$}{pi, beta and boundary}}
Since $P$ is a $2$-type and $c$ is $3$-connected, $c^*\colon[P,P]_\bullet\to[M,P]_\bullet$ is a bijection. Every $f\in\Aut^+_\bullet(M)$ is over $B\pi$, so there is a unique $\varphi_f\in\Aut_\bullet(P)$ with $\varphi_fc\simeq cf$, and
\[
\pi\colon\Aut^+_\bullet(M)\to\Aut_\bullet(P),\qquad\pi(f)=\varphi_f,
\]
is a homomorphism. For $\varphi\in\Aut_\bullet(P)$ put
\[
\beta(\varphi)=[M,\varphi c,\sigma]-[M,c,\sigma]\in\Omega^B_4(P).
\]
For $f\in\Aut^+_\bullet(M)$ we have $[M,\varphi_fc,\sigma]=[M,cf,\sigma]=c_*[M,f,\sigma]$, hence
\begin{equation}\label{eq:betapi}
\beta(\pi(f))=c_*\alpha(f).
\end{equation}
The relative group $\Omega^B_5(P,M)$ consists of bordism classes of compact $5$-dimensional $B$-manifolds $(W,F)$ over $P$ with $F(\partial W)\subset M$, and $\partial[W,F]=[\partial W,F|_{\partial W}]\in\Omega^B_4(M)$. The pair $(P,M)$ gives a long exact sequence in the homology theory $\Omega^B_*$. Put $\widehat\Omega^B_5(P,M)=\partial^{-1}\bigl(\widehat\Omega^B_4(M)\bigr)$.

\begin{lemma}\label{lem:bexact}
The sequence $\Omega^B_5(M)\to\Omega^B_5(P)\to\widehat\Omega^B_5(P,M)\xrightarrow{\ \partial\ }\widehat\Omega^B_4(M)\to\Omega^B_4(P)$ is exact.
\end{lemma}
\begin{proof}
It is the long exact sequence of the pair $(P,M)$ in which $\Omega^B_4(M)$ and $\Omega^B_5(P,M)$ are replaced by the subgroups $\widehat\Omega^B_4(M)$ and $\partial^{-1}\widehat\Omega^B_4(M)$. Exactness at $\Omega^B_5(P)$ is unchanged. At $\widehat\Omega^B_5(P,M)$ the kernel of $\partial$ is the image of $\Omega^B_5(P)$, which lies in the subgroup. At $\widehat\Omega^B_4(M)$ the kernel of the map to $\Omega^B_4(P)$ is $\widehat\Omega^B_4(M)\cap\operatorname{im}\partial=\partial\bigl(\partial^{-1}\widehat\Omega^B_4(M)\bigr)$.
\end{proof}

\subsubsection{Surgery and lifting}
Let $V$ be a compact connected $B$-manifold over $P$ of dimension $d\in\{5,6\}$ with $\partial V=A\cup C$, where $A$ and $C$ are compact codimension $0$ submanifolds of $\partial V$ with $A\cap C=\partial A=\partial C$ and $C$ connected, and such that $\Phi|_C$ induces isomorphisms on $\pi_j$ for $j\le2$.

\begin{lemma}\label{lem:surgery}
After surgeries on embedded spheres in the interior of $V$, the manifold $(V,\Phi)$ is replaced by a $B$-bordant manifold $(V',\Phi')$ over $P$, rel boundary, such that $V'$ is connected and $\Phi'_*\colon\pi_j(V')\to\pi_j(P)$ is an isomorphism for $j\le\ell_d$, where $\ell_5=1$ and $\ell_6=2$. Then the pair $(V',C)$ is $\ell_d$-connected, so $H_j(V',C;G)=0$ for $j\le\ell_d$ and every local coefficient system $G$ on $V'$.
\end{lemma}
\begin{proof}
Since $\Phi|_C$ is an isomorphism on $\pi_j$ for $j\le2$, the map $\Phi_*\colon\pi_j(V)\to\pi_j(P)$ is onto for $j\le2$. We kill $\ker\Phi_*$ on $\pi_1$ and, if $d=6$, then on $\pi_2$. For $d=5$ the kernel on $\pi_2$ is not touched, since surgery on a $2$-sphere in a $5$-manifold produces a new dual $2$-sphere. The kernel of $\pi_1(V)\to\pi_1(P)=\pi$ is normally generated by finitely many elements, because $\pi_1(V)$ and $\pi$ are finitely presented. For $d=6$, after this step the universal cover of $V$ is simply connected and compact, because $\pi$ is finite, so $\pi_2(V)=H_2(\widetilde V)$ is a finitely generated module over the Noetherian ring $\Z\pi_1(V)$, and so is the submodule $\ker\Phi_*$. Represent each generator by an embedded sphere $S^j\subset\mathrm{int}\,V$, which is possible because $2j<d$. A sphere in the kernel is orientation preserving, because $w_1(V)=\Phi^*w$ and $\Phi_*$ kills its class. The $B$-structure map $V\to B=B\mathrm{Spin}\times B\pi$ has $B\pi$-component $c_P\Phi$, and $\pi_j(B\mathrm{Spin})=0$ for $j\le3$, so the class of the sphere is zero in $\pi_j(B)$. Hence the stable normal bundle of $V$ is trivial on the sphere, so the normal bundle of the sphere is stably trivial. It is trivial: for $j=1$ because the circle is orientation preserving, and for $j=2$ because rank $d-2$ bundles over $S^2$ are classified by $\pi_1SO(d-2)=\Z/2$, which maps isomorphically to $\pi_1SO$. Since the sphere is null-homotopic in $B$, the $B$-structure extends over the trace of the surgery for a suitable framing; this is Kreck's criterion for surgery below the middle dimension. The map $\Phi$ is null-homotopic on the sphere, so it extends over the trace as well. Since $2j+2\le d$ for $j\le\ell_d$, the dual handle of the surgery has index $d-j\ge j+2$, so the surgery kills the submodule generated by the sphere, creates no element of $\pi_i$ for $i\le j$, leaves $\pi_i$ unchanged for $i<j$ and keeps the manifold connected. After the last step $\Phi'_*$ is an isomorphism on $\pi_j$ for $j\le\ell_d$.

Now $\pi_j(C)\to\pi_j(V')$ is an isomorphism for $j\le\ell_d$, since both groups map isomorphically to $\pi_j(P)$ and the composite is induced by $\Phi|_C$. Hence $\pi_j(V',C)=0$ for $j\le\ell_d$. Replacing $(V',C)$ by a homotopy equivalent relative $CW$ complex without relative cells of dimension at most $\ell_d$, we see that $H_j(V',C;G)=0$ for $j\le\ell_d$ and every local coefficient system $G$.
\end{proof}

Maps are based, and every construction below carries a base arc, that is an arc in $V$ joining the base points of the boundary pieces, on which the maps are prescribed. Adding the arc to $A$ changes $H^k(V,A;G)$ only for $k\le2$, whereas all obstruction groups below occur in degrees at least $3$, so the arc does not enter the arguments and we do not mention it again.

\begin{lemma}\label{lem:lifting}
Let $(V,\Phi)$ be as in \cref{lem:surgery} with $d\in\{5,6\}$, with $\Phi_*$ an isomorphism on $\pi_j$ for $j\le\ell_d$, and let $\Phi_A\colon A\to M$ be a map with $c\Phi_A\simeq\Phi|_A$. Then, after changing $\Phi$ by a homotopy so that $c\Phi_A=\Phi|_A$, there is a map $r\colon V\to M$ with $r|_A=\Phi_A$ and a homotopy $cr\simeq\Phi$ that is stationary on $A$.
\end{lemma}
\begin{proof}
The fibre $X$ of $c$ is $2$-connected, hence simple, and the obstructions to the existence of $r$ lie in $H^{i+1}(V,A;\pi_i(X))$ for $i\ge3$. The coefficients are local systems pulled back from $P$. By Poincar\'e--Lefschetz duality for the manifold $V$ with orientation system $\Phi^*\Zw$,
\[
H^k(V,A;G)\cong H_{d-k}(V,C;G\otimes\Phi^*\Zw).
\]
So the obstruction groups are $H_{d-i-1}(V,C;G_i)$ with $i\ge3$ and $G_i=\pi_i(X)\otimes\Phi^*\Zw$. The degree is at most $d-4$, which is $1=\ell_5$ for $d=5$ and $2=\ell_6$ for $d=6$, so the groups vanish by \cref{lem:surgery}.
\end{proof}

\subsubsection{Normal form for relative classes}
Say that $(W,F)$ is in \emph{standard form} if $W$ is connected, $\partial W=\partial_1W\sqcup N$ with $\partial_1W=-M$ and $F|_{\partial_1W}=c$, and $g:=F|_N\colon N\to M$ has twisted degree $1$. Here $N$ is a closed $B$-manifold over $M$, possibly disconnected.

\begin{lemma}\label{lem:normalform}
Every class in $\widehat\Omega^B_5(P,M)$ has a representative in standard form. If two standard form representatives $(W,F)$ and $(W',F')$ with boundaries $-M\sqcup N$ and $-M\sqcup N'$ define the same class, then there is a $B$-manifold $T$ over $P$ of dimension $6$ with
\[
\partial T=W\cup W'\cup(M\times I)\cup Y_N,
\]
where $M\times I$ joins $\partial_1W$ to $\partial_1W'$, $Y_N$ is a $B$-bordism over $M$ between $N$ and $N'$, and the map of $T$ to $P$ extends $F$, $F'$, $c\circ\mathrm{pr}_1$ on $M\times I$ and the map $Y_N\to M$.
\end{lemma}
\begin{proof}
Let $[W_0,F_0]\in\widehat\Omega^B_5(P,M)$, so that $[\partial W_0,F_0|_{\partial W_0}]$ has twisted degree $0$. The class of $(M\times I,c\circ\mathrm{pr}_1)$ in $\Omega^B_5(P,M)$ is zero: it is the part $M\times I\times\{0\}$ of the boundary of $M\times I\times I$, and the rest of that boundary maps to $M$. So $W_1=W_0\sqcup(M\times I)$ represents the same class. Put $\partial_1W_1=-M\times\{0\}$ and $N_1=\partial W_0\sqcup M\times\{1\}$ with the map $F_0|_{\partial W_0}\sqcup\mathrm{id}$; its twisted degree is $0+1=1$. Join the components of $W_1$ by boundary connected sums along arcs ending in $N_1$ and by interior connected sums with closed components. This does not change the class and keeps $\partial_1W_1$ and the degree.

For the second assertion, a relative bordism between the two representatives is a $B$-manifold $T$ over $P$ with $\partial T=W\cup Y\cup W'$, where $Y$ is a bordism over $M$ between $\partial W=-M\sqcup N$ and $\partial W'=-M\sqcup N'$. In $\Omega^B_4(M)$ we have $[-M,\mathrm{id}]+[N,g]=[-M,\mathrm{id}]+[N',g']$, and cancelling gives $[N,g]=[N',g']$, so there is a bordism $Y_N$ over $M$ between $N$ and $N'$. Put $Y'=(M\times I)\sqcup Y_N$. The closed manifold $\Sigma=Y\cup(-Y')$ over $M$ satisfies $[W\cup Y'\cup W']=[W\cup Y\cup W']-[\Sigma]=-[\Sigma]$ in $\Omega^B_5(P)$, because $T$ is a null-bordism of $W\cup Y\cup W'$. Replacing $Y_N$ by $Y_N\sqcup\Sigma$, which is again a bordism over $M$ between $N$ and $N'$, we get $[W\cup Y'\cup W']=0$, and a null-bordism of this closed manifold is the required $T$.
\end{proof}

\subsubsection{The groups \texorpdfstring{$\widetilde{\mathcal H}(M)$ and $\widetilde{\mathcal H}(P)$}{H(M) and H(P)}}
Let $\iota\colon\Omega^B_5(M)\to\widetilde{\mathcal H}(M)$ be the map of \cref{prop:Hgroup}, whose image is the kernel of $\widetilde{\mathcal H}(M)\to\Aut^+_\bullet(M)$, and let $\widetilde{\mathcal H}(M)\to\Aut^+_\bullet(M)$ be the end map. By \cref{prop:Hgroup} the first sequence of \cref{thm:braidNO} is exact. It remains to treat the other three.

\begin{lemma}\label{lem:HeMP}
The map $c_\sharp\colon\widetilde{\mathcal H}(M)\to\widetilde{\mathcal H}(P)$, $(W,F)\mapsto(W,cF)$, is a bijection.
\end{lemma}
\begin{proof}
\emph{Well defined.} If $f=F|_{\partial_2W}\in\Aut^+_\bullet(M)$, then $cf=\varphi_fc$ is a $3$-equivalence over $B\pi$. A homotopy $h$ in $M$ gives the homotopy $ch$, and $c$ maps a null-bordism over $M$ to one over $P$.

\emph{Surjective.} Let $(W,F)\in\widetilde{\mathcal H}(P)$ with $F|_{\partial_2W}=\varphi c$. We may assume $W$ connected, joining components by boundary sums. Apply \cref{lem:surgery} with $d=5$, $A=\partial_1W$ and $C=\partial_2W$; note that $F|_C=\varphi c$ is an isomorphism on $\pi_j$ for $j\le2$. The trace of a surgery rel boundary is a $B$-cobordism rel boundary, so the closed manifold in the equivalence relation of $\widetilde{\mathcal H}(P)$ is unchanged in $\Omega^B_5(P)$, and we may assume $F_*$ is an isomorphism on $\pi_j$ for $j\le\ell_5=1$. By \cref{lem:lifting} with $\Phi_A=\mathrm{id}_M$ there is a lift $\widehat F\colon W\to M$ of $F$ with $\widehat F|_{\partial_1W}=\mathrm{id}_M$. Put $\widehat f=\widehat F|_{\partial_2W}$. Then $(W,\widehat F)$ is a $B$-bordism over $M$ between $(M,\mathrm{id},\sigma)$ and $(M,\widehat f,\sigma)$, so $\widehat f$ has twisted degree $1$. Moreover $c\widehat f\simeq\varphi c$ is an isomorphism on $\pi_j$ for $j\le2$ and induces the identity on $\pi$, so $\widehat f_*=\mathrm{id}$ on $\pi_1$. By \cref{lem:PDup}, $\widehat f$ is a homotopy equivalence. Thus $(W,\widehat F)\in\widetilde{\mathcal H}(M)$, and $c_\sharp(W,\widehat F)=(W,c\widehat F)$ is equivalent to $(W,F)$ by the homotopy $c\widehat F\simeq F$.

\emph{Injective.} Let $(W,F),(W',F')\in\widetilde{\mathcal H}(M)$ have equivalent images. Then there is a based homotopy $h_P\colon cf\simeq cf'$ in $P$ such that $Z_P=-W'\cup W\cup(M\times I)$ with the map $cF'\cup cF\cup h_P$ bounds a $B$-manifold $(V,\Phi)$ over $P$. Replacing disjoint unions by connected sums, we may assume $W$, $W'$ and $V$ connected. Put $A=W\cup W'$, glued along $\partial_1=M$, and $C=M\times I$, so $\partial V=A\cup C$ and $\dim V=6$. The map $\Phi|_A$ has the lift $F\cup F'$ to $M$, and $\Phi|_C$ is $c$ composed with $h_P$, an isomorphism on $\pi_j$ for $j\le2$. By \cref{lem:surgery} we may assume $\Phi_*$ is an isomorphism on $\pi_j$ for $j\le\ell_6=2$. By \cref{lem:lifting} there is a lift $\widehat\Phi\colon V\to M$ extending $F\cup F'$. Its restriction $h=\widehat\Phi|_C\colon M\times I\to M$ is a based homotopy from $f$ to $f'$, and $(V,\widehat\Phi)$ is a null-bordism over $M$ of $-W'\cup W\cup(M\times I)$ with the map $F'\cup F\cup h$. So $(W,F)\sim(W',F')$.
\end{proof}

Let $\theta\colon\Omega^B_5(P)\to\widetilde{\mathcal H}(M)$ be the composite of $[N,g]\mapsto(M\times I\sqcup N,\ c\,\mathrm{pr}_1\sqcup g)\in\widetilde{\mathcal H}(P)$ with the inverse of $c_\sharp$. Then $\theta\circ c_*=\iota$. Let $\kappa\colon\pi_1(\mathcal E_\bullet(P))\to\Omega^B_5(P)$ send a loop $\lambda$ of based self-equivalences of $P$ at $\mathrm{id}$, with adjoint $\widehat\lambda\colon P\times S^1\to P$, to $[M\times S^1,\ \widehat\lambda\circ(c\times\mathrm{id})]$. Here $M\times S^1$ carries the $B$-structure obtained from $\sigma\times I$ by identifying the two ends, that is the product of $\sigma$ with the Lie group spin structure of $S^1$. It is a $B$-manifold over $P$ because $c\widehat\lambda\simeq c\,\mathrm{pr}$ canonically by \cref{lem:canon}.

\begin{lemma}\label{lem:exactB}
The sequence of pointed sets
\[
\pi_1(\mathcal E_\bullet(P))\xrightarrow{\ \kappa\ }\Omega^B_5(P)\xrightarrow{\ \theta\ }\widetilde{\mathcal H}(M)\to\Aut_\bullet(P)\xrightarrow{\ \beta\ }\Omega^B_4(P)
\]
is exact, where $\widetilde{\mathcal H}(M)\to\Aut_\bullet(P)$ sends $(W,F)$ to $\varphi$ with $F|_{\partial_2W}\simeq\varphi c$. Moreover, for $x,x'\in\Omega^B_5(P)$ we have $\theta(x)=\theta(x')$ if and only if $x-x'\in\operatorname{im}\kappa$.
\end{lemma}
\begin{proof}
By \cref{lem:HeMP} we may work with $\widetilde{\mathcal H}(P)$. The map to $\Aut_\bullet(P)$ is well defined because a homotopy $f\simeq f'$ gives $\varphi=\varphi'$.

\emph{Exactness at $\Aut_\bullet(P)$.} If $(W,F)\mapsto\varphi$, then $W$ is a $B$-bordism over $P$ from $(M,c,\sigma)$ to $(M,\varphi c,\sigma)$, so $\beta(\varphi)=0$. Conversely, a bordism over $P$ from $(M,c,\sigma)$ to $(M,\varphi c,\sigma)$ is an element of $\widetilde{\mathcal H}(P)$ mapping to $\varphi$.

\emph{Exactness at $\widetilde{\mathcal H}(P)$.} The composite $\Omega^B_5(P)\to\widetilde{\mathcal H}(P)\to\Aut_\bullet(P)$ is trivial. If $(W,F)\mapsto\mathrm{id}_P$, then $f=F|_{\partial_2W}\simeq c$ by a based homotopy $h$, and $Z=W\cup(M\times I)$ with the map $F\cup h$ is a closed $B$-manifold over $P$. As in the proof of \cref{prop:Hgroup}, $(W,F)$ is equivalent to the image of $[Z,F\cup h]$.

\emph{Exactness at $\Omega^B_5(P)$ and the last assertion.} Let $x=[N,g]$ and $x'=[N',g']$. By the definition of equivalence in $\widetilde{\mathcal H}(P)$, $\theta(x)=\theta(x')$ if and only if there is a based homotopy $h$ from $c$ to $c$, that is a based loop $h$ in $\mathrm{Map}_\bullet(M,P)$ at $c$, such that $N\sqcup(-N')\sqcup(M\times S^1)$ with the maps $g$, $g'$ and $h$ is null-bordant over $P$, that is $x-x'=-[M\times S^1,h]$. Since $P$ is a $2$-type and $c$ is $3$-connected, $c^*\colon\mathrm{Map}_\bullet(P,P)\to\mathrm{Map}_\bullet(M,P)$ is a weak equivalence, so every loop $h$ at $c$ is $c^*\lambda$ for a loop $\lambda$ of based self-equivalences of $P$ at $\mathrm{id}$, and $-[M\times S^1,h]=[M\times S^1,\bar h]$ for the reversed loop. Hence $\theta(x)=\theta(x')$ if and only if $x-x'\in\operatorname{im}\kappa$. Taking $x'=0$ gives exactness at $\Omega^B_5(P)$.
\end{proof}

\subsubsection{The map \texorpdfstring{$\gamma$}{gamma}}
Let $\xi\in\widehat\Omega^B_5(P,M)$. By \cref{lem:normalform} choose a representative $(W,F)$ in standard form, with $\partial W=\partial_1W\sqcup N$, $\partial_1W=-M$, $F|_{\partial_1W}=c$ and $g=F|_N$ of twisted degree $1$. Apply \cref{lem:surgery} with $d=5$, $A=N$ and $C=\partial_1W$. This replaces $(W,F)$ by a bordant manifold rel boundary, hence represents the same class $\xi$, and we may assume $F_*$ is an isomorphism on $\pi_j$ for $j\le1$. By \cref{lem:lifting} there is a lift $r\colon W\to M$ of $F$ extending $g$ on $N$. Put
\[
\gamma(\xi)=[\,r|_{\partial_1W}\,]\in[M,M]_\bullet .
\]

\begin{lemma}\label{lem:gammawd}
The class $\gamma(\xi)$ is a well defined element of $\Aut^+_\bullet(M)$, independent of the representative, the surgeries and the lift, and $\pi(\gamma(\xi))=1$.
\end{lemma}
\begin{proof}
Put $f=r|_{\partial_1W}$. Since $cf\simeq F|_{\partial_1W}=c$ over $B\pi$, the map $f$ induces the identity on $\pi$ and isomorphisms on $\pi_i$ for $i\le2$. The manifold $(W,r)$ is a $B$-bordism over $M$ between $(N,g)$ and $(M,f)$, so $f$ has the same twisted degree as $g$, namely $1$. By \cref{lem:PDup}, $f$ is a homotopy equivalence, so $f\in\Aut^+_\bullet(M)$, and $\pi(f)=1$ because $cf\simeq c$.

For independence, let $(W,F,r)$ and $(W',F',r')$ be two such choices for the same class $\xi$. \Cref{lem:normalform} gives a $B$-manifold $T$ over $P$ with $\partial T=W\cup W'\cup(M\times I)\cup Y_N$, on which $W$, $W'$ and $Y_N$ carry the lifts $r$, $r'$ and the map $Y_N\to M$; these agree on $N$ and $N'$. Put $A=W\cup W'\cup Y_N$ and $C=M\times I$. By \cref{lem:surgery} (here $\dim T=6$ and $\Phi|_C=c\,\mathrm{pr}_1$) and \cref{lem:lifting} there is a lift $\widehat\Phi\colon T\to M$ extending the given map on $A$. Its restriction to $C=M\times I$ is a based homotopy from $r|_{\partial_1W}$ to $r'|_{\partial_1W'}$. The same argument with $W=W'$, $T=W\times I$ and different lifts $r,r'$ shows independence of the lift.
\end{proof}

\begin{lemma}\label{lem:alphagamma}
$\alpha\circ\gamma=\partial\colon\widehat\Omega^B_5(P,M)\to\widehat\Omega^B_4(M)$.
\end{lemma}
\begin{proof}
With $(W,F,r)$ as in the definition of $\gamma$ we have $\partial\xi=[\partial W,F|_{\partial W}]=[N,g]-[M,\mathrm{id}]$. The lift $r$ is a bordism over $M$ between $(N,g)$ and $(M,f)$ with $f=\gamma(\xi)$, so $[N,g]=[M,f]$ in $\Omega^B_4(M)$, and $\alpha(\gamma(\xi))=[M,f]-[M,\mathrm{id}]=\partial\xi$.
\end{proof}

Let $j\colon\Omega^B_5(P)\to\widehat\Omega^B_5(P,M)$ be the map of \cref{lem:bexact}, and let $\rho\colon\Omega^B_5(P)\to\Aut^+_\bullet(M)$ be the composite of $\theta$ with the end map.

\begin{lemma}\label{lem:invcom}
$\gamma(j(x))=\rho(x)^{-1}$ for all $x\in\Omega^B_5(P)$.
\end{lemma}
\begin{proof}
Let $x=[N,g]$ and let $W_0$ be the interior connected sum of $M\times I$ and $N$ with the map $c\,\mathrm{pr}_1\sharp g$. After surgery in the interior as in \cref{lem:surgery}, with $C=M\times\{0\}$, we obtain $(W,F)$ with $F_*$ an isomorphism on $\pi_j$ for $j\le1$. It represents $\theta(x)$ in $\widetilde{\mathcal H}(P)$, and in standard form, with $\partial_1W=M\times\{0\}$, $N'=M\times\{1\}$ and $g'=\mathrm{id}_M$, it represents $j(x)$.

By the construction of the bijection of \cref{lem:HeMP}, $\rho(x)=[\widehat f]$, where $\widehat F\colon W\to M$ is a lift of $F$ extending $\mathrm{id}_M$ on $M\times\{0\}$ (it exists by \cref{lem:lifting}) and $\widehat f=\widehat F|_{M\times1}$. Since $F|_{M\times1}=c$, we have $c\widehat f\simeq c$, so $\widehat f$ is a homotopy equivalence in $\ker\pi$, and $c\widehat f^{-1}\simeq c\widehat f\widehat f^{-1}\simeq c$. Put $r_0=\widehat f^{-1}\circ\widehat F$. Then $cr_0\simeq c\widehat F\simeq F$, and $r_0|_{M\times1}=\widehat f^{-1}\widehat f\simeq\mathrm{id}_M$. After a homotopy we may assume that $r_0$ restricts to $\mathrm{id}_M$ on $M\times\{1\}$. The resulting homotopy $cr_0\simeq F$ restricts on $M\times\{1\}$ to a loop in $\mathrm{Map}(M,P)$ based at $c$. Since $\widehat f$ is a homotopy equivalence, precomposition with $\widehat f^{-1}$ is a bijection on $\pi_1$ of the mapping space, so changing the homotopy $c\widehat f^{-1}\simeq c$ by a suitable loop makes this loop trivial. Then $r_0$ is a lift of $F$ extending $\mathrm{id}_M$ on $M\times\{1\}$ in the sense of the definition of $\gamma$, and by \cref{lem:gammawd} the class $\gamma(j(x))$ may be computed with it. Hence $\gamma(j(x))=[r_0|_{M\times0}]=[\widehat f^{-1}]=\rho(x)^{-1}$.
\end{proof}

\begin{corollary}\label{cor:im}
The sequence of pointed sets
\[
\pi_1(\mathcal E_\bullet(P))\xrightarrow{\ j\circ\kappa\ }\widehat\Omega^B_5(P,M)\xrightarrow{\ \gamma\ }\Aut^+_\bullet(M)\xrightarrow{\ \pi\ }\Aut_\bullet(P)
\]
is exact. In particular $\ker\pi=\operatorname{im}\gamma$.
\end{corollary}
\begin{proof}
\emph{Exactness at $\Aut^+_\bullet(M)$.} By \cref{lem:gammawd}, $\operatorname{im}\gamma\subset\ker\pi$. Let $f\in\ker\pi$, so there is a based homotopy $h\colon M\times I\to P$ from $c$ to $cf$. Then $(M\times I,h)$ is in standard form with $\partial_1=M\times\{0\}$, $N=M\times\{1\}$ and $g=f$, which has twisted degree $1$, so it defines a class $\xi\in\widehat\Omega^B_5(P,M)$. By \cref{lem:lifting} with $A=N$ and $C=\partial_1$ (no surgery is needed, since $h$ is an isomorphism on $\pi_j$, $j\le2$) there is a lift $r\colon M\times I\to M$ of $h$ with $r|_{M\times1}=f$. Since $cr\simeq h$, the map $r$ is a homotopy in $M$ from $r|_{M\times0}$ to $f$, so $\gamma(\xi)=[f]$.

\emph{Exactness at $\widehat\Omega^B_5(P,M)$.} If $\xi=j(\kappa(\lambda))$, then $\theta(\kappa(\lambda))$ is the base point by \cref{lem:exactB}, so $\rho(\kappa(\lambda))=\mathrm{id}$ and $\gamma(\xi)=\mathrm{id}$ by \cref{lem:invcom}. Conversely let $\gamma(\xi)=\mathrm{id}$. By \cref{lem:alphagamma}, $\partial\xi=\alpha(\mathrm{id})=0$, so $\xi=j(x)$ for some $x\in\Omega^B_5(P)$ by \cref{lem:bexact}, and $\rho(x)=\mathrm{id}$ by \cref{lem:invcom}. By \cref{prop:Hgroup}, $\theta(x)=\iota(z)=\theta(c_*z)$ for some $z\in\Omega^B_5(M)$. By \cref{lem:exactB}, $x-c_*z=\kappa(\lambda)$ for some $\lambda$, hence $\xi=j(x)=j(\kappa(\lambda))+j(c_*z)=j(\kappa(\lambda))$, since $j\circ c_*=0$ by \cref{lem:bexact}.
\end{proof}

\begin{proof}[Proof of \cref{thm:braidNO}]
The four sequences are exact by \cref{prop:Hgroup}, \cref{lem:exactB}, \cref{lem:bexact} and \cref{cor:im}. For the commutativity of the diagram: $\theta\circ c_*=\iota$ by definition. The square formed by $\widetilde{\mathcal H}(M)\to\Aut^+_\bullet(M)\xrightarrow{\pi}\Aut_\bullet(P)$ and $\widetilde{\mathcal H}(M)\to\Aut_\bullet(P)$ commutes since $\varphi_fc\simeq cf$. The square formed by $\alpha$, $\beta$, $\pi$ and $\widehat\Omega^B_4(M)\to\Omega^B_4(P)$ commutes by \eqref{eq:betapi}. The relation $\alpha\circ\gamma=\partial$ is \cref{lem:alphagamma}, and the two maps $\Omega^B_5(P)\to\Aut^+_\bullet(M)$ through $\gamma\circ j$ and through $\widetilde{\mathcal H}(M)$ are inverse to each other by \cref{lem:invcom}. The remaining triangles commute by definition.
\end{proof}

\begin{remark}\label{rem:finite}
Finiteness of $\pi$ is used in three places: in \cref{lem:degw}, which identifies the twisted degree with the degree of the lift to the universal cover, in \cref{lem:PDup}, and in \cref{lem:surgery}, where it gives finite generation of the kernels. We do not discuss infinite fundamental groups.
\end{remark}

\begin{remark}\label{rem:sminus}
The braid concerns $\Aut^+_\bullet(M)$ only. If $s(f)=-1$ then $f_*[M,\id,\sigma]$ has twisted degree $-1$ and cannot be bordant to $[M,\id,\sigma]$ over $M$.
Spin structures on $\nu_M\ominus c^*E$ form a torsor over $H^1(M;\Z/2)$, and orientation reversal
$\hat\tau$ is an automorphism of $B\mathrm{Spin}$ over $BO$ with $[N,\hat\tau\sigma,g]=-[N,\sigma,g]$. An element with $s(f)=-1$ is realized by a diffeomorphism
only if $f_*[M,\id,\sigma]=[M,\id,\sigma']$ for some $\sigma'\in\{\hat\tau\sigma+x:x\in H^1(M;\Z/2)\}$ (for $M=\RP^2\times S^2$ there are two such $x$, and $x$ is the class by which $f^*$ moves
structures); we use this only as a necessary condition. The torsor does not enter the exactness statements of \cref{thm:braidNO},
but it enters their \emph{use}: if $f$ is a diffeomorphism with $f^*\sigma=\sigma+x$, then $f\colon(M,\sigma+x)\to(M,\sigma)$ is an isomorphism of $B$-manifolds over $M$ (with $f$ as the map), so
$f_*[M,\id,\sigma+x]=[M,\id,\sigma]$ and therefore
$\alpha(f)=f_*\bigl([M,\id,\sigma]-[M,\id,\sigma+x]\bigr)$. Hence such an $f$ lies in $\ker\alpha$ if and only if
$[M,\id,\sigma]=[M,\id,\sigma+x]$ in $\widehat\Omega^B_4(M)$.
\end{remark}

\begin{remark}\label{rem:Pinform}
The braid of \cref{thm:braidNO} is the braid for $\mathrm{Pin}^\pm$-bordism of $X\langle w_1\rangle$, $X=M,P$, where $X\langle w_1\rangle=X\times_{K(\pi,1)}B\mathrm{Pin}^\pm$ is a $B\mathrm{Spin}$-fibration over $X$ (not the orientation double cover). Indeed $\Omega^{\mathrm{Pin}^\mp}_n(X\langle w_1\rangle)=\Omega^B_n(X)$ with $E=L$ for tangential $\mathrm{Pin}^-$ (that is, $\Omega^{\mathrm{Pin}^-}_n=\pi_n(M\mathrm{Spin}\wedge\Sigma^{-1}MO(1))$) and $E=3L$ for tangential $\mathrm{Pin}^+$. In the notation of this paper $\RP^2\times S^2$ is tangential $\mathrm{Pin}^-$ and normal $\mathrm{Pin}^+$, so $E=L$.

\medskip\noindent\emph{The $d_2$ differential.} In the Atiyah--Hirzebruch spectral sequence $E^2_{p,q}=H_p(X;\Omspin_q\otimes\Zw)$ the differential $d_2\colon E_{p,1}\to E_{p-2,2}$ is dual to $D(x)=Sq^2x+w_1(E)\,Sq^1x+w_2(E)\,x$. It is the twisting bundle $E$ whose $w_2$ enters, not $w_2(M)$. For $E=L$ one has $w_2(E)=0$, and
$D(x)=Sq^2x+a\,Sq^1x$. In $M=\RP^2\times S^2$, with $H^2(M;\Z/2)=\langle a^2,b\rangle$, $a^3=0$, $b^2=0$, we get $D(a^2)=D(b)=0$. Hence both classes of $E^2_{2,2}(M)=H_2(M;\Z/2)=\langle[\RP^2],[S^2]\rangle$ survive: $E^\infty_{2,2}(M)=(\Z/2)^2$, in agreement with \cref{prop:AHSS}. (The term $w_2(E)\,x=a^2x$, which is what one obtains by using $w_2(M)=a^2$, belongs to the twist $E=3L$ and would kill $[S^2]$ through $D(b)=a^2b\ne0$.)

\medskip\noindent\emph{The same term in $P$.} Since $M$ is a product, $P=P_2(M)=P_2(\RP^2)\times\mathbb{C}P^\infty$. For the generator $c\in H^2(\mathbb{C}P^\infty;\Z/2)$ we have $D(c)=Sq^2c=c^2\ne0$, so $d_2\colon E_{4,1}(P)\to E_{2,2}(P)$ kills the class $[\mathbb{C}P^1]$, that is the image of $[S^2]$. The class $[\RP^2]$ is not hit by $d_2$, because $a^3=0$ already in $P_2(\RP^2)$ and so $D(a^2)=0$ there. So the discrepancy between $E_{2,2}(M)$ and $E_{2,2}(P)$ is the class $[S^2]\otimes\eta^2$. This class is, up to the sign conventions of \cref{prop:detect}, $\Theta(\kappa_2)=\alpha(k_{\kappa_2})$, and it maps to zero in $\Omega^B_4(P)$ because $c\circ\kappa_2\colon S^4\to P$ is null, as $\pi_4(P)=0$. In the braid this is the statement that $\alpha(\ker\pi)\subseteq\ker(\widehat\Omega^B_4(M)\to\Omega^B_4(P))=\partial\widehat\Omega^B_5(P,M)$ (third sequence of \cref{thm:braidNO}): the difference between $E^\infty_{2,2}(M)$ and $E^\infty_{2,2}(P)$ is what detects $\ker\bigl(\Autb(M)\to\Autb(P)\bigr)$.
\end{remark}

\subsection{Application to \texorpdfstring{$M=\RP^2\times S^2$}{M = RP2 x S2}}\label{sec:applyRP2}
Here $\pi=\Z/2$, $E=L$, $c\colon M\to B\Z/2$ is the projection to $\RP^2\subset\RP^\infty$ and
$w_2(E)=0$, $w_1(E)=a$ (\cref{ex:RP2S2B}).

\begin{proposition}\label{prop:AHSS}
In the twisted Atiyah--Hirzebruch spectral sequence $E^2_{p,q}=H_p(M;\Omspin_q\otimes\Zw)\Rightarrow\Omega^B_{p+q}(M)$,
all differentials leaving or entering total degree $4$ vanish, and $|\widehat\Omega^B_4(M)|=16$.
The $E^\infty$ terms in total degree $4$ are
\[
\begin{array}{c|ccccc}
(p,q)&(4,0)&(3,1)&(2,2)&(1,3)&(0,4)\\\hline
E^\infty&\Z&\Z/2&(\Z/2)^2&0&\Z/2
\end{array}
\]
\end{proposition}
\begin{proof}
The $E^2$-terms are computed from $H_*(M;\Z/2)$, of ranks $1,1,2,1,1$, and $H_0(M;\Zw)=\Z/2$,
$H_4(M;\Zw)=\Z$ (\cref{prop:homotopy}). The differentials $d^2$ are dual to
$D(x)=Sq^2x+Sq^1x\cdot a+x\cdot w_2(E)$ with $w_2(E)=0$ (the $d^2$ formula for twisted spin bordism, see \cite{T93}). In $H^*(M;\Z/2)=\Z/2[a,b]/(a^3,b^2)$ one has
$Sq^2a=0$, $Sq^2a^2=a^4=0$, $Sq^2b=b^2=0$, $Sq^1(a^2)=0=Sq^1b$, $a\cdot Sq^1a=a^3=0$, so $D\equiv0$
on $H^1,H^2$ and on $H^0$. Thus every $d^2$ in the relevant range is zero. Differentials into $E_{0,4}$ and out of
$E_{4,0}$ are handled by the retraction $r\colon M\to\RP^2$ and the inclusion $i$: $E_{0,4}$ in the spectral
sequence of $\RP^2$ cannot be hit, since the only possible sources $E_{2,3},E_{3,2},E_{4,1}$ vanish for dimensional
reasons or because $\Omspin_3=0$, so by naturality of $r_*$ it survives in $M$. The fundamental class gives a permanent cycle
generating $E_{4,0}$, because $(M,\bar\nu)$ itself is a $B$-manifold of twisted degree $1$. For the
remaining terms in total degree $4$, the differentials $d^2\colon E_{3,1}\to E_{1,2}$ and $d^2\colon E_{4,1}\to E_{2,2}$ are dual to $D$ on $H^1$ and $H^2$ and vanish, and every other differential into or out of these terms has source or target zero, because $\Omspin_3=0$, $E_{1,3}=0$ and $E_{p,q}=0$ for $p>4$. Taking the kernel of the twisted degree removes $E_{4,0}$ and leaves a group of order $2\cdot4\cdot2=16$.
\end{proof}

\begin{proposition}\label{prop:detect}
Let $j\colon S^2\to M$, $j(y)=(x_0,y)$, let $\kappa_2=j\circ\eta_2\eta_3$ be a generator of the second summand of $\pi_4(M)$ and let
$\kappa_1=q\circ\eta_2\eta_3$ be a generator of the first, where $q\colon S^2\to\RP^2\subset M$ is the double cover. Then in $\widehat\Omega^B_4(M)$:
\begin{enumerate}[label=(\roman*)]
  \item $\alpha(k_{\kappa_2})\ne0$, and $\alpha(k_{\kappa_2})$ is represented by $[S^2]\otimes\eta^2\ne0$ in $E^\infty_{2,2}$;
  \item $\alpha(k_{\kappa_1})=[K3\to\mathrm{pt}]\ne0$, the generator of $E^\infty_{0,4}=\Z/2$.
\end{enumerate}
\end{proposition}
\begin{proof}
For $\kappa\in\pi_4(M)$ the pinch-fold map satisfies, in $\Omega^B_4(M)$,
\begin{equation}\label{eq:cutpaste}
  [M,\bar\nu,k_\kappa]=[M,\bar\nu,\id]+[S^4,\kappa],
\end{equation}
where $S^4$ carries its unique spin structure and the trivial twist. This is the usual cut and paste:
$k_\kappa$ is the identity outside a disc, and $M$ with the reference map $k_\kappa$ is bordant over $M$
to $M\sqcup S^4$. Put $\Theta(\kappa)=[S^4,\kappa]$, so that $\alpha(k_\kappa)=\Theta(\kappa)$.

(i) The map $j$ is covered by a map of twisted spectral
sequences from that of $S^2$ (untwisted, since $c\circ j$ is constant), and $j_*\colon H_2(S^2;\Z/2)\to H_2(M;\Z/2)$ is injective.
In $\Omega^{\mathrm{spin}}_4(S^2)=\Omega^{\mathrm{spin}}_4\oplus\Omega^{\mathrm{spin}}_2$ the class $x=[S^4,\eta_2\eta_3]$ is $[S^2]\otimes\eta^2$: the preimage of a regular value of
$\eta_2\eta_3$ is a torus with the Lie group framing, whose class generates $\pi_2^s\cong\Omspin_2=\Z/2$.
By naturality $\Theta(\kappa_2)=j_*x$ is represented by a non-zero element of $E^\infty_{2,2}(M)=E^2_{2,2}(M)$ (\cref{prop:AHSS}).

(ii) Write $\Theta(\kappa_1)=q_*x$ for $q_*\colon\Omega^{\mathrm{spin}}_4(S^2)\to\Omega^B_4(\RP^2)\to\Omega^B_4(M)$. The map $q_*$ is zero on $E_{2,2}$, because $q_*[S^2]=2[\RP^2]=0\in H_2(\RP^2;\Z/2)$, so
$\Theta(\kappa_1)$ lies in filtration $0$ of $\Omega^B_4(\RP^2)$, that is in $\{0,[K3]\}$, where $E^\infty_{0,4}(\RP^2)=\Z/2$ survives
because its possible sources $E_{2,3},E_{3,2},E_{4,1}$ vanish for dimensional reasons or since $\Omspin_3=0$; the same holds in $M$ by the retraction $M\to\RP^2$.
The cofibration $S^2=S(L)\to\RP^2\to\mathrm{Th}(L)$ gives a Gysin sequence
\[
  \Omega^{(2L)}_4(\RP^2)\xrightarrow{\ \partial\ }\Omega^{\mathrm{spin}}_4(S^2)\xrightarrow{\ q_*\ }\Omega^B_4(\RP^2),
\]
where $\Omega^{(2L)}$ is spin bordism twisted by $2L$ and $\partial$ sends $(N\to\RP^2)$ to the pulled back double cover with its induced structure.
Let $Y$ be an Enriques surface, with universal cover $K3$ and classifying map $f_\infty\colon Y\to\RP^\infty$ of this cover. It is orientable, not spin, and its $w_2$ lies in the kernel of $H^2(Y;\Z/2)\to H^2(K3;\Z/2)$, which is $\{0,a^2\}$ by the Gysin sequence.
Hence $w_2(Y)=a^2=w_2(2L)$, so $Y$ carries a $2L$-twisted spin structure and defines a class in $\Omega^{(2L)}_4(\RP^\infty)$. This class lies in the image of $\Omega^{(2L)}_4(\RP^2)$. Indeed, in the Atiyah--Hirzebruch spectral sequence for $\RP^\infty$ we have $E^2_{4,0}=H_4(\RP^\infty;\Z)=0$, and $E^2_{3,1}=H_3(\RP^\infty;\Z/2)$ supports the non-zero differential $d^2$ dual to $D(a)=Sq^2a+w_2(2L)\,a=a^3\ne0$. So every class in $\Omega^{(2L)}_4(\RP^\infty)$ has skeletal filtration at most $2$. Hence there is a closed $2L$-twisted spin manifold $N$ with a map $g\colon N\to\RP^2$ that represents the same class in $\Omega^{(2L)}_4(\RP^\infty)$ as $Y$. In particular $N$ is oriented with $\sigma(N)=\sigma(Y)=-8$.
The image of $[N,g]$ under $\partial$ is $[\widetilde N,\widetilde g]$, where $\widetilde N\to N$ is the pulled back double cover with its induced spin structure and $\widetilde g\colon\widetilde N\to S^2$ is the lift of $g$. Since $\sigma(\widetilde N)=2\sigma(N)=-16=\sigma(K3)$ and $K3$ generates $\Omspin_4$,
$[\widetilde N,\widetilde g]=K3_{\mathrm{pt}}+\varepsilon x$ for some $\varepsilon\in\{0,1\}$. Exactness gives $0=q_*[\widetilde N,\widetilde g]=[K3]+\varepsilon\Theta(\kappa_1)$ in $\Omega^B_4(\RP^2)$. As $[K3]\ne0$ and $\Theta(\kappa_1)\in\{0,[K3]\}$, this forces $\varepsilon=1$ and $\Theta(\kappa_1)=[K3]\ne0$.
\end{proof}

\begin{remark}\label{rem:k1}
\Cref{prop:kernel} (A) and \cref{prop:detect} (ii) are two independent proofs that $k_{\kappa_1}\ne\id$. The element $\alpha(k_{\kappa_1})$ is the class of a K3 surface in $E^\infty_{0,4}=\Z/2$, a term through which $\alpha$ does not pass in the oriented spin case, where the kernel of $\pi$ is detected in $E_{2,2}$ \cite[Proposition~4.5 and Corollary~4.9]{HK04}. For simply connected manifolds the pinch maps are detected by their normal invariants \cite[Theorems~2.2 and~5.1]{CH90}.
Here the normal invariant does not detect $k_{\kappa_1}$ through its $H^2$-component. Indeed, by Hypothesis~\ref{hyp:pinch} (\cref{sec:hyp}) the $H^2$-component of the normal invariant of the pinch map $k_{y\circ\eta^2}$ is $\mathrm{PD}\,h(y)$, where $h\colon\pi_2(M)\to H_2(M;\Z/2)$ is the mod~$2$ Hurewicz map and $\mathrm{PD}$ is the Poincar\'e duality isomorphism $H_2(M;\Z/2)\to H^2(M;\Z/2)$. For the class $y_1\in\pi_2(M)$ coming from the double cover $q$ we have $h(y_1)=q_*[S^2]=2[\RP^2]=0$, whereas $h(y_2)=[S^2]\ne0$ for the $S^2$-factor. So, assuming Hypothesis~\ref{hyp:pinch}, the non-trivial self-equivalence $k_{\kappa_1}$ has vanishing $H^2$-component of its normal invariant.
\end{remark}

\begin{remark}
With $E_{\infty}$ as in \cref{prop:AHSS}, $|\widehat\Omega^B_4(M)|=16$ while $|\Aut^+_\bullet(M)|=16$. The map $\alpha$
is a crossed homomorphism (\cref{prop:crossed}), not a homomorphism, and on the central
subgroup $K=\langle k_{\kappa_1},k_{\kappa_2}\rangle$ it is additive and injective (\cref{prop:detect}). Whether $ab$ and $s_\psi$ are seen by $\alpha$ is not determined here:
by \cref{rem:sminus} this depends on the class $[M,\id,\sigma]-[M,\id,\sigma+a]$, and $\ker\alpha$ is the group of self-equivalences that the braid
realizes by $B$-cobordisms.
\end{remark}

\subsection{The braid for \texorpdfstring{$\RP^4$}{RP4}}\label{sec:RP4}
Let $M=\RP^4$, $\pi=\Z/2$, $P=P_2(M)=\RP^\infty$ since $\pi_2(\RP^4)=0$.
We have $w(\nu_M)=(1+a)^{-5}=(1+a)^3$, so $w_2(\nu_M)=a^2=w_1^2$; $\wt M=S^4$ is spin,
and $E=3L$ (indeed $w(3L)=(1+a)^3$). So $M$ has normal $\mathrm{Pin}^-$ type, equivalently tangential $\mathrm{Pin}^+$, and
$\Omega^B_n(P)=\Omega^{\mathrm{Pin}^-}_n$ is the normal $\mathrm{Pin}^-$ bordism group:
$\Z/2,0,\Z/2,\Z/2,\Z/16,0$ in degrees $0,\dots,5$ (\cref{sec:preliminaries}).

\begin{proposition}\label{prop:RP4}
(i) $\widehat\Omega^B_4(\RP^4)\cong\Z/8$, and $c_*\colon\widehat\Omega^B_4(M)\to\Omega^B_4(P)=\Z/16$ is injective with image the subgroup of index $2$.
(ii) $\widehat\Omega^B_5(P,M)=0$.
(iii) $\Aut^+_\bullet(\RP^4)=1$.
\end{proposition}
\begin{proof}
(i) The pair $(P,M)=(\RP^\infty,\RP^4)$ is $4$-connected, so $\Omega^B_n(P,M)=0$ for $n\le4$ and $\Omega^B_5(P,M)\cong H_5(P,M;\Zw)$.
In the cellular chain complex of $\RP^\infty$ with coefficients $\Zw$ the boundary $d_n\colon C_n\to C_{n-1}$ is $0$ for $n$ even and
multiplication by $2$ for $n$ odd. Hence $H_5(P,M;\Zw)=\mathrm{coker}(d_6)=\Z$, and likewise $H_4(M;\Zw)=\Z$,
$H_4(P;\Zw)=\Z/2$. The exact sequence of the pair together with $\Omega^B_5(P)=0$ gives
\[
0\to\Z=\Omega^B_5(P,M)\to\Omega^B_4(M)\to\Omega^B_4(P)=\Z/16\to0,
\]
and the image of the first term in $H_4(M;\Zw)=\Z$ is $d_5\Z=2\Z$, since the degree of the boundary of the $5$-cell is $d_5=2$.
It follows that $\widehat\Omega^B_4(M)\cap\Omega^B_5(P,M)=0$, so $c_*$ is injective on $\widehat\Omega^B_4(M)$. The degree is compatible
with $H_4(M;\Zw)=\Z\to H_4(P;\Zw)=\Z/2$ and the latter is detected by the epimorphism $\Z/16\to\Z/2$ (it is the lowest filtration of the
Atiyah--Hirzebruch spectral sequence, and $E^\infty_{4,0}(P)=\Z/2$ because $[\RP^4]$ is a $B$-manifold over $P$), so the image lies in its kernel $\Z/8$. Conversely, given $x\in\Z/8\subset\Z/16$ take a lift
$y\in\Omega^B_4(M)$; its degree is even, equal to $2m$, and subtracting $m$ times the generator of $\Omega^B_5(P,M)$ (degree $2$) gives a lift of $x$ of degree $0$.
(ii) The third sequence of \cref{thm:braidNO} reads $0=\Omega^B_5(P)\to\widehat\Omega^B_5(P,M)\xrightarrow{\partial}\widehat\Omega^B_4(M)\to\Omega^B_4(P)$.
By (i) the last map is injective, so $\im\partial=0$, and $\partial$ is injective because $\Omega^B_5(P)=0$.
(iii) The fourth sequence gives $\widehat\Omega^B_5(P,M)=0\to\Aut^+_\bullet(M)\to\Aut_\bullet(\RP^\infty)=\Aut(\Z/2)=1$.
\end{proof}

\subsection{Homeomorphisms and diffeomorphisms homotopic to pinch-fold maps}\label{sec:hyp}
The statements about homeomorphisms below depend on the following two facts, which are taken as hypotheses. Each is a result from the literature whose proof for the
case at hand is not written out, or is stated in the literature for a different generality.
\begin{hypothesis}\label{hyp:pinch}
Let $M$ be a closed non-orientable $4$-manifold with $\pi_1(M)=\Z/2$ and $w_2(\wt M)=0$. For $y\in\pi_2(M)$ the $H^2(M;\Z/2)$-component of the normal invariant of the pinch map $k_{y\circ\eta^2}$ in $[M,G/\mathrm{TOP}]\cong H^2(M;\Z/2)\oplus H^4(M;\Z)$ is $\mathrm{PD}\,h(y)$, where $h\colon\pi_2(M)\to H_2(M;\Z/2)$ is the mod $2$ Hurewicz map and $\mathrm{PD}\colon H_2(M;\Z/2)\to H^2(M;\Z/2)$ is Poincar\'e duality. For simply connected manifolds this is \cite[Theorem~5.1 and its proof]{CH90}, where the condition $w_2(\widetilde M)=0$ is replaced by $\langle w_2(X),h(y)\rangle=0$; this condition holds here, since $y$ lifts to $\widetilde M$. It is also the statement used in \cite[Remark~1.3]{HH23}, where $|\pi_1|=4$ and $h=0$.
\end{hypothesis}
\begin{hypothesis}\label{hyp:sigma}
For $M=\RP^2\times S^2$ the surgery obstruction $\sigma\colon[M,G/\mathrm{TOP}]\to L_4(\Z/2,-)=\Z/2$ is onto. (Here $L_4(\Z/2,-)\cong\Z/2$ is detected by the Arf invariant \cite[proof of Lemma~6]{HKT94}. On normal invariants we expect the obstruction to be $\sigma(x,y)=\langle x\cup w_1^2,[M]\rangle$, under which $b\mapsto\langle b\,a^2,[M]\rangle\ne0$.)
\end{hypothesis}
\begin{proposition}\label{prop:nodiffeo}
Let $M=\RP^2\times S^2$ with the fixed $B$-structure $\sigma$. For a diffeomorphism $\varphi$ of $M$ let $\varphi^*\sigma$ be the pull-back of $\sigma$ along $d\varphi$, using the canonical identification $c^*E\cong\varphi^*c^*E$ of \cref{lem:canon}. No diffeomorphism $\varphi$ with $\varphi^*\sigma=\sigma$ represents a non-trivial element of $K=\ker\pi$.
\end{proposition}
\begin{proof}
Let $\varphi$ be a diffeomorphism with $\varphi^*\sigma=\sigma$ and $[\varphi]\in K$. Then $\varphi$ is an isomorphism of $B$-manifolds over $M$ from $(M,\varphi,\sigma)$ to $(M,\mathrm{id},\sigma)$, so $[M,\varphi,\sigma]=[M,\mathrm{id},\sigma]$ and $\alpha(\varphi)=0$. By \cref{lem:transport}, $\alpha$ depends only on the based homotopy class. For $\kappa\in\pi_4(M)$ we have $\alpha(k_\kappa)=\Theta(\kappa)$ by \eqref{eq:cutpaste}, and $\Theta$ is additive. By \cref{prop:detect}, $\Theta(\kappa_1)$ is a non-zero class in filtration $0$, and $\Theta(\kappa_2)$ is non-zero in $E^\infty_{2,2}$. Hence $\Theta(\kappa)\ne0$ for every $\kappa\ne0$, and $[\varphi]=k_\kappa$ forces $\kappa=0$.
\end{proof}

\begin{remark}\label{rem:Delta}
The condition $\varphi^*\sigma=\sigma$ in \cref{prop:nodiffeo} is needed. If $\varphi$ is a diffeomorphism with $\varphi^*\sigma=\sigma+a$, then by \cref{rem:sminus} one has
$\alpha(\varphi)=\varphi_*\Delta$ with $\Delta=[M,\id,\sigma]-[M,\id,\sigma+a]\in\widehat\Omega^B_4(M)$, so a diffeomorphism representing $k_{\kappa_1}$ can exist only if
$\Delta=[K3]$.
\end{remark}

\begin{proposition}\label{prop:Q6M}
Let $\mathcal S(M)=\mathcal S^{\mathrm{TOP}}(M)$ be the topological structure set, and $\mathcal H\le\Autb(M)$ the subgroup of classes
homotopic to a homeomorphism. Assume Hypothesis~\ref{hyp:sigma} for (i)--(iii) and Hypothesis~\ref{hyp:pinch} for (iv). Then
\begin{enumerate}[label=(\roman*)]
 \item $\ast(\RP^2\times S^2)\simeq\RP^2\times S^2$, and $|\mathcal S(M)|=4$ under Hypothesis~\ref{hyp:sigma};
 \item $[\Autb(M):\mathcal H]=2$, so $|\mathcal H|=16$;
 \item $\mathcal H\cap K$ has order $2$, where $K=\langle k_{\kappa_1},k_{\kappa_2}\rangle=\ker\pi$, and $\mathcal H\supseteq S=\langle a,b,s_\psi\rangle$;
 \item $\mathcal H\cap K=\langle k_{\kappa_1}\rangle$, so $k_{\kappa_2}$ and $k_{\kappa_1}k_{\kappa_2}$ are not homotopic to homeomorphisms, and the $H^2(M;\Z/2)$-component of the normal invariant of $k_{\kappa_2}$ is $a^2$.
\end{enumerate}
\end{proposition}
\begin{proof}
(i) The homotopy equivalence $\ast M\simeq M$ is part of the construction in (ii), and it also follows from \cite[Corollary~1]{HKT94}, since both manifolds have $w_2$-type~II, $e=2$ and $w_1^4=0$. Since $L_5(\Z/2,-)=0$ the normal invariant map
$\eta\colon\mathcal S(M)\to[M,G/\mathrm{TOP}]\cong H^2(M;\Z/2)\oplus H^4(M;\Z)$ is injective with image $\ker\sigma$, and
$H^2(M;\Z/2)=\langle a^2,b\rangle\cong(\Z/2)^2$, $H^4(M;\Z)=\Z/2$ (twisted Poincar\'e duality gives $H^4(M;\Z)\cong H_0(M;\Zw)$). So $|[M,G/\mathrm{TOP}]|=8$ and $|\mathcal S(M)|=4$ when $\sigma$ is onto.

(ii) The twin $\ast N'$ of a structure (Teichner's star construction \cite{T97}) $[N',g]$ is obtained from the normal map $N'\#E_8\to N'$, whose normal invariant is the class $e$ of the collapse $N'\to S^4\to G/\mathrm{TOP}$ and whose surgery obstruction vanishes because $L_4(1)\to L_4(\Z/2,-)$ is zero \cite[p.~651]{HKT94}. Hence $\ast$ is a free involution on $\mathcal S(M)$ which reverses $\mathrm{KS}$, and $\mathcal S_0=\{\mathrm{KS}=0\}$ has $2$ elements. By \cite[Theorem~3]{HKT94} every closed manifold in the homotopy type of $M$ with $\mathrm{KS}=0$ is homeomorphic to $M$ (the entries of type~II with $e=2$ are $M$ and $\ast M$ only). If $[N',g]\in\mathcal S_0$ and $\varphi\colon M\to N'$ is a homeomorphism then $[N',g]=[M,g\varphi]$ with $g\varphi\in\Autb(M)$. So $\Autb(M)$ acts transitively on $\mathcal S_0$ through $f\cdot[N',g]=[N',fg]$, and the stabiliser of $[M,\id]$ is $\mathcal H$. Hence $[\Autb(M):\mathcal H]=|\mathcal S_0|=2$.

(iii) $a,b,s_\psi$ are diffeomorphisms, so $S\le\mathcal H$ and $S\cap K=1$. Since $|\Autb(M)|=32$ and $\Autb(M)=S\times K$ we get $|\mathcal H|=16=|\mathcal H\cap K|\cdot|\mathcal H K/K|\le|\mathcal H\cap K|\cdot8$ and $\mathcal H\ne\Autb(M)$, so $|\mathcal H\cap K|=2$.

(iv) Let $\kappa=y\circ\eta^2$ and $s=[M,k_\kappa]\in\mathcal S(M)$. Every element of $K$ has this form with $y=0$, $y_1$, $y_2$ or $y_1+y_2$, since $\pi_4(M)=\pi_2(M)\circ\eta^2$ and $[y_1,y_2]=0$. By Hypothesis~\ref{hyp:pinch} the $H^2$-component of $\eta(s)$ is $\mathrm{PD}\,h(y)$. Now $h(y_1)=q_*[S^2]=0$ for the class $y_1$ coming from the double cover, and $h(y_2)=[S^2]$ for the $S^2$-factor, with $\mathrm{PD}[S^2]=a^2$. If $h(y)=0$, then $\eta(s)\in\{\eta(s_0),\eta(\ast s_0)\}$, since these are the two elements of $\eta(\mathcal S(M))$ with vanishing $H^2$-component. The domain of $s$ is $M$ and has $\mathrm{KS}=0$, whereas the domain of $\ast s_0$ has $\mathrm{KS}=1$. So $s=s_0$ by injectivity of $\eta$, that is $k_\kappa\in\mathcal H$. If $h(y)\ne0$, the $H^2$-component of $\eta(s)$ is $a^2$, so $s\ne s_0$ and $k_\kappa\notin\mathcal H$.
\end{proof}

\begin{remark}
For the statement $\mathcal H\cap K=\langle k_{\kappa_1}\rangle$ there are two routes, the counting statement \cref{prop:Q6M}(iii) and the normal invariants of \cref{prop:Q6M}(iv).
The counting route forces $|\mathcal H\cap K|=2$, and if $k_{\kappa_1}\in\mathcal H$ then it forces $k_{\kappa_2},k_{\kappa_1}k_{\kappa_2}\notin\mathcal H$ without any use of $\alpha(k_{\kappa_2})$.
So $\mathcal H\cap K=1$ is excluded by Hypothesis~\ref{hyp:sigma} and \cite[Theorem~3]{HKT94} alone, whereas distinguishing the three subgroups of order $2$ in $K$ requires Hypothesis~\ref{hyp:pinch}, which by (iv) suffices for this purpose without Hypothesis~\ref{hyp:sigma}.
In the language of the structure set, $\mathcal S(M)=\{s_0,\ k_{\kappa_2}\cdot s_0,\ \ast s_0,\ \ast(k_{\kappa_2}\cdot s_0)\}$, where $s_0=[M,\id]$, and $k_{\kappa_2}$ is the pinch-fold map detected in $E^\infty_{2,2}$.
\end{remark}

\section{The Hambleton--Hillman manifold \texorpdfstring{$N$}{N}}
\label{sec:manifoldN}

Hambleton and Hillman \cite{HH23} construct a closed non-orientable
$4$-manifold $N$ with $\pi_1(N)\cong(\Z/2)^2$, $\chi(N)=1$ and universal cover $S^2\times S^2$. It is the quotient $N=(\RP^4\#_{S^1}\RP^4)/F$ by a free
involution $F$, has the quadratic $2$-type of $\RP^2\,\widetilde\times\,\RP^2$, and is not
homotopy equivalent to a geometric quotient of $S^2\times S^2$. Its orientation double cover is
$S(\eta\oplus2\varepsilon)$. The manifold $N$ has two non-orientable double
covers, namely $\RP^4\#_{S^1}\RP^4$, which has $v_2\ne0$, and a second one with $v_2=0$ which by
\cite[Remark~9.2]{HH23} is homotopy equivalent to $\RP^2\times S^2$. In this section we record the facts needed to apply
\cref{thm:braidNO} to $N$ and determine the action of self-equivalences on $\pi_2$.

Write $\pi=\langle t,u\mid t^2=u^2=(tu)^2=1\rangle$ and $t^*,u^*$ for the dual basis of $H^1(\pi;\Z/2)=H^1(N;\Z/2)$.
By \cite[Proposition~7.2, Remark~9.2, Proposition~9.1 and the proof of Theorem~9.3]{HH23}, $w_1(N)=t^*+u^*$, $v_2(N)=t^*u^*+(u^*)^2$, the classes satisfy $(u^*)^3=0\ne(t^*)^3$, the universal cover is homotopy equivalent to $S^2\times S^2$ and
$\pi_2(N)=\Z_t\oplus\Z_u$ with $\Z_t=\Lambda/(t+1,u-1)$, $\Z_u=\Lambda/(t-1,u+1)$, where $\Lambda=\Z\pi$. Moreover $N$ is homotopy equivalent to
$Y_{\eta_1'}=K'\cup_{[J']+\eta_1'}e^4$ with $K'=\RP^2\widetilde\times\RP^2\smallsetminus D^4$.

\paragraph{The normal \texorpdfstring{$1$}{1}-type.} Since $w_2(\nu_N)=v_2(N)$, we need $E\to B\pi$ with $w_1(E)=t^*+u^*$ and $w_2(E)=t^*u^*+(u^*)^2$. The bundle
$E=L_t\oplus3L_u$ works, where $L_t,L_u$ are the real line bundles with $w_1=t^*,u^*$: $w(E)=(1+t^*)(1+u^*)^3=1+(t^*+u^*)+(t^*u^*+(u^*)^2)+\cdots$. So $B=B\mathrm{Spin}\times B\pi$ with $\xi=\gamma\oplus E$ is a normal $1$-type of $N$, the universal cover is spin, and \cref{thm:braidNO} applies to $N$.

\paragraph{The torsor of \texorpdfstring{$B$}{B}-structures.} The set of $B$-structures on $\nu_N$ with prescribed $B\pi$-component and fixed orientation of $\nu_N\ominus c^*E$ is a torsor over $H^1(N;\Z/2)=(\Z/2)^2$, so it has four elements, and the opposite orientation gives four more. (For $\RP^2\times S^2$ the numbers are two and two.) Let $\mathrm{Diff}^+(N)$ be the group of diffeomorphisms of $N$ fixing the base point whose based lift to the universal cover has degree $1$. For $f\in\mathrm{Diff}^+(N)$ we have $f^*\sigma=\sigma+x(f)$ for a well defined $x(f)\in H^1(N;\Z/2)$, and since $f^*$ is trivial on $H^1(N;\Z/2)$ the map
\[
x\colon\mathrm{Diff}^+(N)\to(\Z/2)^2,\qquad x(fg)=x(f)+x(g),
\]
is a homomorphism, which we call the spin character. It depends on the diffeomorphism and not only on its class in $\Autp(N)$. By \cref{rem:sminus}, a diffeomorphism $f$ with $x(f)\ne0$ lies in $\ker\alpha$ if and only if $[N,\id,\sigma]=[N,\id,\sigma+x(f)]$ in $\widehat\Omega^B_4(N)$.

\begin{proposition}\label{prop:AutPN}
Every $f\in\Autb(N)$ induces the identity on $\pi$ and $\pm1$ on each of $\Z_t$ and $\Z_u$. If $(\varepsilon_t,\varepsilon_u)$ are the signs, then the degree of the lift of $f$ to the universal cover is $s(f)=\varepsilon_t\varepsilon_u$.
\end{proposition}
\begin{proof}
A based self-equivalence $f$ induces an automorphism $f_1$ of $\pi$ preserving $w_1(N)=t^*+u^*$. The automorphisms of $\pi$ are the permutations of the three non-trivial elements, and those preserving $t^*+u^*$ (the character with kernel $\langle tu\rangle$) fix $tu$, so $f_1\in\{\id,\,t\leftrightarrow u\}$. Since $f_1^*$ is a ring automorphism of $H^*(N;\Z/2)$ and $(t^*)^3\ne0=(u^*)^3$, the swap is excluded. Thus $f_1=\id$ and $f_{2}=\pi_2(f)$ is $\Lambda$-linear. As $\mathrm{Hom}_\Lambda(\Z_t,\Z_u)=\mathrm{Hom}_\Lambda(\Z_u,\Z_t)=0$, we get $f_2=\varepsilon_t\oplus\varepsilon_u$ with $\varepsilon_t,\varepsilon_u\in\{\pm1\}$, because $\mathrm{End}_\Lambda(\Z_t)=\mathrm{End}_\Lambda(\Z_u)=\Z$ and an automorphism of $\Z$ is $\pm1$. The lift of $f$ to the universal cover $\widetilde N\simeq S^2\times S^2$ acts on $H_2(\widetilde N)=\Z_t\oplus\Z_u$, whose intersection form $\lambda$ is hyperbolic with $\Z_t$ and $\Z_u$ isotropic, by $\mathrm{diag}(\varepsilon_t,\varepsilon_u)$. (For a generator $e$ of $\Z_t$ the deck transformation $t$ acts by $-1$ and reverses the orientation of $\widetilde N$, so $\lambda(e,e)=\lambda(te,te)=-\lambda(e,e)$; the same argument applies to $\Z_u$ and $u$.) For a homotopy equivalence $\widetilde f$ one has $\lambda(\widetilde f_*x,\widetilde f_*y)=\deg(\widetilde f)\,\lambda(x,y)$, and $\lambda(e_t,e_u)=\pm1$ for generators $e_t,e_u$, so $s(f)=\varepsilon_t\varepsilon_u$.
\end{proof}

\section*{Acknowledgements}
I used an AI language model (Claude, Anthropic) 
to edit the manuscript’s language, check LaTeX, and review selected mathematical arguments.  I have verified all of the mathematics and take full responsibility for the content.


\begin{thebibliography}{HKT94}

\bibitem[BM25]{BM25}
  R.\,I.\ Baykur and P.\ Morgan,
  \textit{On nonorientable 4-manifolds},
  Preprint, arXiv:2506.20950, 2025.

\bibitem[Fre82]{Freed82}
  M.\,H.\ Freedman,
  \textit{The topology of four-dimensional manifolds},
  J.\ Differential Geom.\ \textbf{17} (1982), no.\ 3, 357--453.

\bibitem[CH90]{CH90}
  T.\,D.\ Cochran and N.\ Habegger,
  \textit{On the homotopy theory of simply connected four manifolds},
  Topology \textbf{29} (1990), no.\ 4, 419--440.

\bibitem[FQ90]{FQ90}
  M.\,H.\ Freedman and F.\ Quinn,
  \textit{Topology of 4-Manifolds},
  Princeton Mathematical Series \textbf{39}, Princeton University Press, 1990.

\bibitem[HH23]{HH23}
  I.\ Hambleton and J.\,A.\ Hillman,
  \textit{Quotients of $S^2\times S^2$},
  J.\ London Math.\ Soc.\ (2)\ \textbf{108} (2023), no.\ 4, 1393--1416.
  \href{https://doi.org/10.1112/jlms.12783}{doi:10.1112/jlms.12783}

\bibitem[HK04]{HK04}
  I.\ Hambleton and M.\ Kreck,
  \textit{Homotopy self-equivalences of 4-manifolds},
  Math.\ Z.\ \textbf{248} (2004), no.\ 1, 147--172.
  \href{https://doi.org/10.1007/s00209-004-0657-9}{doi:10.1007/s00209-004-0657-9}

\bibitem[HKT94]{HKT94}
  I.\ Hambleton, M.\ Kreck, and P.\ Teichner,
  \textit{Nonorientable 4-manifolds with fundamental group of order 2},
  Trans.\ Amer.\ Math.\ Soc.\ \textbf{344} (1994), no.\ 2, 649--665.
  \href{https://doi.org/10.1090/S0002-9947-1994-1234481-2}{doi:10.1090/S0002-9947-1994-1234481-2}

\bibitem[T93]{T93}
  P.\ Teichner,
  \textit{On the signature of four-manifolds with universal covering spin},
  Math.\ Ann.\ \textbf{295} (1993), 745--759.

\bibitem[T97]{T97}
  P.\ Teichner,
  \textit{On the star-construction for topological 4-manifolds},
  in \textit{Geometric Topology (Athens, GA, 1993)},
  AMS/IP Stud.\ Adv.\ Math.\ \textbf{2.2},
  Amer.\ Math.\ Soc., Providence, RI, 1997, pp.\ 300--312.

\bibitem[Tor26]{Tor26}
  R.\ Torres,
  \textit{Topological classification of certain nonorientable 4-manifolds
  with cyclic fundamental group of order 2$\bmod 4$},
  Preprint, arXiv:2601.05132, 2026.

\end{thebibliography}
\end{document}